\documentclass[11pt,leqno]{article}

\usepackage{amsthm,amsfonts,amssymb,amsmath,oldgerm,mathtools}
\usepackage{fullpage}
\usepackage{graphicx,subfig}
\usepackage{mathrsfs}
\usepackage{xcolor}
\usepackage{tikz}
\usepackage[colorlinks=true]{hyperref}
\usepackage[margin=1in]{geometry}
\usepackage[margin=1in]{caption}

\def\be{\begin{equation}}
\def\ee{\end{equation}}
\def\bse{\begin{subequations}}
\def\ese{\end{subequations}}

\numberwithin{equation}{section}

\usepackage{mathrsfs}

\renewcommand\d{\partial}

\renewcommand\th{\theta}

\def\eps {\varepsilon}
\def\half{{\textstyle\frac12}}
\newcommand{\R}{\mathbb R}
\newcommand{\C}{\mathbb C}

\newcommand{\RM}{{\mathbb{R}}}
\newcommand{\CM}{{\mathbb{C}}}

\newcommand{\ZM}{{\mathbb{Z}}}
\def\half{{\textstyle\frac12}}
\def\thalf{{\textstyle\frac32}}
\newtheorem{theorem}{Theorem}[section]
\newtheorem{proposition}[theorem]{Proposition}

\newtheorem{lemma}[theorem]{Lemma}
\newtheorem{remark}[theorem]{Remark}
\theoremstyle{definition}
\newtheorem{definition}[theorem]{Definition}

\allowdisplaybreaks[3]
\title{Modulational Stability of Wave Trains in the Camassa-Holm Equation}

\author{Mathew~A.~Johnson\thanks{Department of Mathematics, University of Kansas, 1460 Jayhawk Boulevard, 
Lawrence, KS 66045, USA; matjohn@ku.edu}\quad \&\quad
Jeffrey Oregero\thanks{Department of Mathematics, University of Kansas, 1460 Jayhawk Boulevard, Lawrence, KS 66045, USA; oregero@ku.edu}}

\date{\today}

\begin{document}

\maketitle

%%%%%%%%%%%%%%%%%%%%%%%%%%%%%%%%%%%%%%%%%%%%%%%%%
%%%%%%%%%%%%%%%%%%%%%%%%%%%%%%%%%%%%%%%%%%%%%%%%%
\begin{abstract}
In this paper, we study the nonlinear wave modulation of arbitrary amplitude periodic traveling wave solutions of the Camassa-Holm (CH) equation.  Slow modulations
of wave trains is often described through Whitham's theory of modulations, which at leading order models the slow evolution of the fundamental wave characteristics 
(such as the wave's  
frequency, mass and momentum) through a disperionless system of quasi-linear partial differential equations.  The modulational stability or instability of such a slowly modulated 
wave is considered to be determined by the hyperbolicity or ellipticity of this Whitham modulation system of equations.  
%For the CH equation, the Whitham modulation equations have been derived through an averaged Lagrangian method and were subsequently shown to not always be hyperbolic, indicating that the CH equation admits both modulationally stable and modulationally unstable solutions.  
In work by Abenda \& Grava, the Whitham modulation system for the CH equation was derived through averaged Lagrangian methods and was further shown to always
be hyperbolic (although strict hyperbolicity may fail).
In this work, we provide an independent derivation of the Whitham modulation system for the CH equation through nonlinear WKB / multiple scales expansions.  
We further provide 
a rigorous connection between the Whitham modulation equations for the CH equation
and the spectral stability of the underlying periodic wave train to localized (i.e. integrable on the line) perturbations.
In particular, %for the CH equation we prove that the ellipticity of the Whitham modulation system implies the spectral instability of the underlying wave train, while 
we prove that the strict hyperbolicity of the Whitham system implies spectral stability in a neighborhood of the origin in the spectral plane, i.e. spectral
modulational stability.  As an illustration of our theory, we examine the Whitham modulation system for wave trains with asymptotically small oscillations about their total mass.
\end{abstract}
	
%%%%%%%%%%%%%%%%%%%%%%%%%%%%%%%%%%%%%%%%%%%%%%%%%
%%%%%%%%%%%%%%%%%%%%%%%%%%%%%%%%%%%%%%%%%%%%%%%%%

\section{Introduction}

We consider the stability of periodic traveling wave solutions of the Camassa-Holm (CH) equation
\begin{equation}
\label{e:ch}
u_t-u_{txx}+3uu_x=2u_xu_{xx}+uu_{xxx},
\end{equation}
which was derived\footnote{The CH equation had actually been derived more than a decade earlier \cite{FF81} as a bi-Hamiltonian generalization
of the KdV equation. } in \cite{CH93,CHH94} and later justified as a model for the propagation of unidirectional surface water waves over a flat bottom in \cite{CL09,Johnson_RS02}.
The CH equation is strongly nonlinear and is known to exhibit a number of water wave phenomena, such as wave breaking \cite{CE98}, 
that are not modeled in more conventional shallow-water models such as the Korteweg-de Vries (KdV) and Benjamin-Bona-Mahony (BBM) equations.  The CH equation
has been the object of much study, and is known to be bi-Hamiltonian as well as completely integrable via the Inverse Scattering Transform \cite{Const_01}.
Additionally, the CH equation is well-known to admit large classes of both smooth and peaked traveling solitary and periodic waves, as well as mult-soliton type solutions.  The dynamics and stability of both smooth and peaked solutions has been the subject of many works: see, for example, \cite{CP21,CW02,GMNP22}.

In this work, we are interested in studying the stability and dynamics of slowly modulated periodic traveling wave solutions, sometimes referred to as wave trains, in the Camassa-Holm
equation.  Slow modulations of periodic surface water waves has a long and rich history, with possibly the first understanding of its importance in water wave theory coming from 
the work by Whitham in the 1960's.  In \cite{W67}, Whitham, applied his newly developed modulation theory to study the stability of small amplitude, periodic Stokes
waves in finte-depth water.  By deriving what  is now referred to as the modulation equations in this context, he showed that a Stokes wave
with frequency $k$ in a fluid with undisturbed depth $h_0$ is unstable if $kh_0>1.363\ldots$.  Around the same time, Benjamin and Feir \cite{BF67}, being motivated by 
laboratory experiments \cite{F67}, showed that a Stokes wave in deep water is unstable and additionaly discovered the so-called ``sideband" nature of the instability.
This immediately spurred a great amount of work, with corroborating findings being independently derived and provided nearly simultaneously by
Lighthill \cite{L65}, Ostrovsky \cite{O67}, Benny and Newell \cite{BN67}, and Zakharov \cite{Z68}.

While the various arguments cited above are indeed correct, they are difficult to justify in a mathematically rigorous, functional analytic, nonlinear PDE setting.
Indeed, it wasn't until the work of Bridges and Mielke \cite{BM95} that such a
rigorous understanding of the modulational, or sideband, instability of asymptotically small amplitude Stokes waves in finite-depth water was provided. 
In their work, they proved
that when the so-called Benjamin-Feir criteria $kh_0>1.363\ldots$ is satisfied then the linearization of the governing PDE
about the associated surface water wave, considered as an operator acting on localized (i.e. integrable
on the line) perturbations, admits curves of unstable spectrum in a sufficiently small neighbood of the origin in the spectral plane.  This provided a first result
connecting the rigorous spectral instability of a periodic wave train to the predictions from Whitham's theory of modulations.  More recently, there have been a number
of similar results pertaining to the modulational instability of asymptoically small amplitude Stokes waves.  See, for instance, the work of Nguyen and Strauss \cite{NS23}
`where they extend the work of Bridges and Mielke to the infinite-depth problem.  See also the works \cite{BMV23,HY23} as well as the very recent works 
\cite{CNS23,CNS24_2,CNS24_1,JRSY24}, all of which consider various aspects of the spectral and modulational instability of Stokes waves.

While the above works all have to do with modulations of Stokes waves, the theory of wave modulation for large amplitude periodic traveling waves (i.e. those outside
of the asymptotically small amplitude limit of the Stokes waves) is considerably less developed.  For these large amplitude waves, current analytical studies are confined
to the setting of various model equations. 
%In these settings one often derives the so-called Whitham equations, which models the 
%slow evolution of the fundamental wave characteristics (such as frequency, mass and momentum) through a dispersionless system of quasi-linear PDE.
As in the case of the aforementioned analysis of Stokes waves, however, a mathematically rigorous justification 
of the predictions of Whitham's theory is quite difficult, even in the studies of simpler model equations.  In recent years, however, various authors have succeeded in providing
rigorous connections between Whitham's theory of modulations and the spectral stability of the underlying large amplitude wave train.  See, for instance,
\cite{BMR21,BrHJ16,JP20,JZ10}.  Most of these studies, however, deal with model equations that are long-wavelength and weakly nonlinear.
We also note that there are several works on the modulational instability of asymptotically small Stokes waves for model equations: see, for instance, 
 \cite{HLS06,Har08,HJ15,HJ15_2,HK08,J13}.  

\

Our goal is to extend the above studies to provide a mathematically rigorous connection between Whitham's theory of wave modulation and the spectral stability
of periodic wave trains for large amplitude waves in the strongly nonlinear Camassa-Holm equation \eqref{e:ch}.  
In particular, we first discuss the existence theory
for periodic traveling wave solutions of the CH equation, showing that such solutions form a 3-dimensional manifold of solutions parameterized by
the spatial frequency $k>0$ as well as the mass $M$ and momentum $P$  (both of which are c`onserved quantities) of the wave.  Using a multiple-scales / WKB approximation,
we then show that slow modulations of a given wave train $\phi_0(x)=\phi(x;k_0,M_0,P_0)$ will formally evolve via
\[
u(x,t;X,T) = \phi(x;k(X,T),M(X,T),P(X,T)),
\]
where $(X,T)=(\eps x,\eps t)$ are the slow space and time variables, with the modulation functions $k,M,P$  evolving slowly near $(k_0,M_0,P_0)$ according
to the a quasi-linear first order system of the form
\begin{equation}\label{e:W_intro}
\partial_T\left(k,M,P\right) = \mathcal{W}(\phi)\partial_X\left(k,M,P\right),
\end{equation}
where here $\mathcal{W}(\phi)$ is a $3\times 3$ matrix depending on the underlying slowly modulated wave.  According to Whitham's modulation theory,
the stability of the underlying wave $\phi_0$ to slow modulations is determined by the eigenvalues of the constant coefficient matrix
$\mathcal{W}(\phi_0)$.  Indeed, linearizing the Whitham system \eqref{e:W_intro} about $\phi_0$, treated as a constant solution in the slow variables $(X,T)$, we see that the eigenvalues
of the linearization will be of the form
\[
\nu_j(\xi)=i\xi\alpha_j,~~j=1,2,3
\]
where the $\alpha_j$ are the eigenvalues of the matrix $\mathcal{W}(\phi_0)$ and $\xi\in\RM$ is a parameter.  As such, hyerbolicity\footnote{Meaning the the eigenvalues
of $\mathcal{W}(\phi_0)$ are all real.} of the Whitham system \eqref{e:W_intro} at $\phi_0$ indicates (marginal) spectral stability of $\phi_0$
as a solution of \eqref{e:W_intro}.  On the other hand, ellipticity\footnote{Meaning $\mathcal{W}(\phi_0)$ has an eigenvalue with
non-zero imaginary part.} of \eqref{e:W_intro} at $\phi_0$ implies the spectral instability of $\phi_0$ 
as a solution of \eqref{e:W_intro}, indicating an instability of $\phi_0$ to slow modulations.  

\begin{remark}\label{r:Grava}
We point out that the Whitham modulation equations for the CH equation have previously been derived in \cite{AG05} by using an averaged Lagrangian method.
There, the authors express the modulation system in terms of its so-called Riemann-invariants which, while algebraically nice does not seemingly allow
for a direct connection to the rigorous spectral stability problem.  As such, here we choose to provide an alternate derivation of the Whitham averaged system
which, as we will see, does allow for a direct connection to a rigorous spectral stability calculation.  
\end{remark}

We then turn our attention to rigorously validating the above predictions from Whtiham's theory of modulations.  In particular, we study the stability
of the underlying periodic traveling wave $\phi_0$, as an equilibrium solution of the CH equation in appropriate traveling coordinates, 
to perturbations in $L^2(\RM)$.  By performing a rigorous, functional analytic analysis of the spectral stability problem in a 
neighborhood of the origin in the spectral plane, we prove that the spectrum of the corresponding linearized operator near the origin
consists of precisely three $C^1$ curves that admit the expansion
\[
\lambda_j(\xi) = \nu_j(\xi)+\mathcal{O}(\xi) = i\xi\alpha_j+\mathcal{O}(\xi),~~|\xi|\ll 1
\]
where the $\alpha_j\in\CM$ are precisely the eigenvalues of the matrix $\mathcal{W}(\phi_0)$ coming from Whitham's theory of modulations.
As such, we establish the following as our main result.

%\begin{theorem}\label{T:main}
%Suppose that $\phi_0$ is a $T_0=1/k_0$-periodic traveling wave solution of \eqref{e:ch} with wave speed $c_0>0$, and that the set of nearby periodic traveling
%wave profiles $\phi$ with speed close to $\phi_0$ is a 3-dimensional smooth manifold parameterized by $(k,M(\phi),P(\phi))$, where $1/k$ denotes the fundamental period
%of the wave and $M$ and $P$ denotes the mass and momentum of the wave.  Under an appropriate non-degeneracy hypothesis\footnote{For a precise
%statement, see Theorem \ref{t:co-per} below.}, a sufficient condition for $\phi_0$ to be spectrally unstable to perturbations in $L^2(\RM)$
%is that the Whitham modulation system \eqref{e:W_intro} be elliptic at $(k_0,M(\phi_0),P(\phi_0))$.  Further, if the Whitham system \eqref{e:W_intro} is strictly 
%hyperbolic\footnote{Meaning the eigenvalues of $\mathcal{W}(\phi_0)$ above are real and distinct.}
%at $(k_0,M(\phi_0),P(\phi_0))$ then the wave $\phi_0$ is spectrally modulationally stable, i.e. the $L^2(\RM)$-spectrum of the linearization of \eqref{e:ch} about $\phi_0$ 
%is purely imaginary in a sufficiently small neighborhood of origin in the spectral plane.
%\end{theorem}

\begin{theorem}\label{T:main}
Suppose that $\phi_0$ is a $T_0=1/k_0$-periodic traveling wave solution of \eqref{e:ch} with wave speed $c_0>0$, and that the set of nearby periodic traveling
wave profiles $\phi$ with speed close to $\phi_0$ is a 3-dimensional smooth manifold parameterized by $(k,M(\phi),P(\phi))$, where $1/k$ denotes the fundamental period
of the wave and $M$ and $P$ denotes the mass and momentum of the wave.  Additionally, assume that the hypotheses of Theorem \ref{t:co-per} below holds.
If the  Whitham system \eqref{e:W_intro} is strictly 
hyperbolic\footnote{Meaning the eigenvalues of $\mathcal{W}(\phi_0)$ above are real and distinct.}
at $(k_0,M(\phi_0),P(\phi_0))$ then the wave $\phi_0$ is spectrally modulationally stable, i.e. the $L^2(\RM)$-spectrum of the linearization of \eqref{e:ch} about $\phi_0$ 
is purely imaginary in a sufficiently small neighborhood of the origin in the spectral plane.  Furthermore, a sufficient condition for $\phi_0$ to be spectrally 
unstable to perturbations in $L^2(\RM)$ is that the Whitham modulation system \eqref{e:W_intro} be elliptic at $(k_0,M(\phi_0),P(\phi_0))$.  
\end{theorem}

\begin{remark}
The  hypothesis from Theorem \ref{t:co-per} mentioned above has to do with the dimensionality and the structure of the generalized kernel of the linearization
of \eqref{e:ch} to co-periodic perturbations, i.e. perturbations that have the same fundamental period as $\phi$.  As we will see, understanding the 
generalized kernel of the linearization is the first major step in our rigorous analysis.  This is discussed in detail in Section \ref{s:unmod}.
\end{remark}

%
%\begin{remark}
%In \cite[Section 5(a)]{HP17}, Hur and Pandey  studied the rigorous spectral stability problem for periodic traveling wave solutions of the CH equation
%to localized perturbations.  There, they conclude through spectral perturbation methods that a $2\pi/k$-periodic Stokes wave with asymptotically
%small oscillations are modulationally unstable when $k>6$ and modulationally stable when $k<6$.  Consequently, Theorem \ref{T:main} implies
%that there are wave trains for the CH equation for which the associated Whitham system is necessarily elliptic (when $k>6$) and hyperbolic (when $k<6$).
%\end{remark}

Importantly, we note that in \cite{AG05} the authors proved that the Whitham modulation system associated to the CH equation, derived there through averaged Lagrangian
techniques\footnote{See Remark \ref{r:Grava}.} is \emph{always} hyperbolic.
That is, the Whitham system for the CH equation is never elliptic.  The system, however, can fail to be \emph{strictly} hyperbolic, indicating a degeneracy in the Whitham system where two 
of (necessarily real) eigenvalues of $\mathcal{W}(\phi)$ in \eqref{e:W_intro} coalesce.  We illustrate this explicitly in Section \ref{s:W_stokes}
by analyzing the Whitham system explicitly for Stokes waves, i.e. for wave trains of CH with asymptotically small oscillations about their mean.  There, we will
see that such  $2\pi/k$-perioidc waves trains with small oscillations are strictly hyperbolic for all $k\neq k_*:= \sqrt{3}/2\pi$.  By Theorem \ref{T:main} it follows
that all such Stokes waves with $k\neq k^*$ are in fact spectrally modulationally stable. For larger amplitude waves, we similarly expect that
the Whitham system is generically strictly hyperbolic, corresponding to the spectral modulational stability of the underlying wave trains.  Regardless, our theory
shows that all wave trains for which the Whitham system is strictly hyperbolic are necessarily spectrally modulationally stable.

\begin{remark}
In the literature, one will sometimes find the Camassa-Holm equation written as
\[
u_t-u_{txx}+3uu_x+2\nu u_x=2u_xu_{xx}+uu_{xxx},
\]
where here $\nu\in\RM$ is a constant parameter.  We note, however, that this linear term can be transformed away via the transformation
\[
u(x,t)\mapsto u(x+\nu t,t)-\nu.
\]
Therefore, without loss of generality, in our work we take $\nu=0$ throughout.
\end{remark}

\

The organization of this paper is as follows.  In Section \ref{s:basic} we discuss some basic properties of the Camassa-Holm equation, such as the existence 
theory for periodic traveling waves and also the Hamiltonian structure and associated conserved quantities which will be important throughout our work.  
In Section \ref{s:W} we derive the Whitham modulation system for the CH equation by a direct multiple-scales / WKB approximation.  We further provide
an explicit study of the Whitham system in the case of Stokes waves for the CH equation, i.e. for periodic traveling wave solutions with asymptotically
small oscillation about their mass.  In Section \ref{s:rigorous}  we then perform a rigorous spectral stability calculation, 
using Floquet-Bloch theory and spectral perturbation theory, to derive a $3\times 3$ matrix
whose eigenvalues rigorously encode the spectrum in a sufficiently small neighbood of the origin in the spectral plane of the associated linearized operator.  
Theorem \ref{T:main} is then established in Section \ref{s:proof} by providing a direct, term-by- term comparison of the Whitham matrix $\mathcal{W}$ discussed above and 
the $3\times 3$ matrix coming from the rigorous theory.  Finally, some technical results pertaining to the analytic parameterization of the Stokes waves
studied in Section \ref{s:W_stokes} are contained in Appendix \ref{a:stokes}.

%
%
%\
%
%
%The CH equation \eqref{e:ch} is an example of a nonlinear dispersive PDE governing the evolution of waves in shallow water when surface tension is present \cite{}. Remarkably, the CH equation \eqref{e:ch} admits a Hamiltonian structure \cite{}, has a Lax pair \cite{}, and is integrable by the inverse scattering trasform (IST) \cite{}.
%
%%
%Note that under the transformation
%\[
%u(x,t)=v(x+kt,t)-k
%\]
%the CH equation becomes
%\[
%v_t-v_{txx}+(??)kv_x+(b+1)vv_x=bv_xv_{xx}+vv_{xxx},
%\]
%which is the version of b-CH we see sometimes... as a result, there is really no reason NOT to study the $k=0$ case.  Also, note you get the classical
%CH equation when $b=2$ and the integrable DP equation when $b=3$.  Throughout we assume $b>1$, although I think this is only used in the existence theory.
%
%\
%
%\noindent
%Goals:
%\begin{itemize}
%\item Derive Whitham: for this, I would follow my old conduit equation work with Wesley.
%\item Do rigorous spectral MI calculation, using Whitham to set various hypotheses (like multiplicity of zero eigenvalue, etc.).  This should be done in the physical coordinates, based
%on my experience.
%\item If Mark is up for it, do numerics on one of the sides so that we get stability plots.
%\end{itemize}

\

\noindent
{\bf Acknowledgments:} The work of MJ was partially supported by the NSF under grant DMS-2108749.  The authors are thankful to Mark Hoefer for making us aware
of the work \cite{GPT09}, and also thank Wesley Perkins for helpful conversations.

\section{Basic Properties of the Camassa-Holm Equation}\label{s:basic}
In this section, we record some important basic results regarding the CH equation \eqref{e:ch}.  We start by recalling the existence theory for periodic
traveling wave solutions of \eqref{e:ch}.  Of particular importance here will be the parameterization of the manifold of such solutions, which will be given
in terms of natural integration constants arising from reducing the associated profile equation to quadrature.  We will then discuss the Hamiltonian structure
and, more importantly, the conserved quantities admitted by \eqref{e:ch} that will be used throughout our work.

\subsection{Existence of Periodic Traveling Waves}
\label{s:existence} 

The existence of smooth periodic traveling wave solutions of the Camassa-Holm equation \eqref{e:ch} has by now been well studied: see, for example,
\cite{EJ24,GMNP22} and references therein.  For completeness, however, we review this existence theory here.

Traveling wave solutions of \eqref{e:ch} correspond to solutions of the form
\be
\label{e:traveling}
u(x,t) = \phi(x-ct),
\ee
where here $\phi$ is the wave profile and $c>0$ the wave speed.  Introducing the traveling coordinate frame $\eta=x-ct$, it follows that the profile $\phi(\eta)$ is necessarily a stationary solution 
of the evolution equation
\be
\label{e:travelframe}
u_t - u_{t\eta\eta} - cu_{\eta} + cu_{\eta\eta\eta} + 3uu_{\eta} = 2u_{\eta}u_{\eta\eta} + uu_{\eta\eta\eta},
\ee 
i.e. $\phi$ should be a solution of the ODE
\[
-c(\phi'-\phi''') + 3\phi\phi' = 2\phi'\phi'' + \phi\phi''',
\] 
or equivalently,
\be
\label{e:profile} 
-(c-\phi)(\phi-\phi'')'+ 2\phi'(\phi-\phi'') = 0,
\ee
where $'$ denotes differentiation with repsect to $\eta$. Here and throughout the work we refer to \eqref{e:profile} as the \textit{profile equation}. Note that, 
by elementary ODE theory, it follows that $\phi\in C^{\infty}(\RM)$ provided
that either $\phi(\eta)<c$, or $\phi(\eta)>c$, for all $\eta\in\R$.  As it can be shown\footnote{This observation follows by similar phase plane methods that we use
to study the case $\phi<c$ below.} that no smooth periodic traveling waves exists satisfying $\phi>c$, 
we will only consider the case when
\be
\label{e:condition}
\phi(\eta) < c, \quad \forall\,\eta\in\R.
\ee 

Equipped with condition \eqref{e:condition}, multiplying \eqref{e:profile} by $(c-\phi)$ and integrating yields
\be
\label{e:profile2}
\phi-\phi''=\frac{a}{(c-\phi)^2},
\ee
where $a\in\RM$ is a  constant of integration. Multiplying by $\phi'$ and integrating again, \eqref{e:profile2} can be reduced to the quadrature representation
\be
\label{e:quad}
\frac{1}{2}\left(\phi'\right)^2=E-\left(-\frac{1}{2}\phi^2+\frac{a}{c-\phi}\right),
\ee
where here $E\in\R$ is a second constant of integration.

%%%%%%%%%%%%%%%%%%%%%%%%%%%%%%%%%%%%%%%%%%%%%%%%%%%%%%%%%%%%%%%%%%%%%%%%%%%%%%%%%%%%%
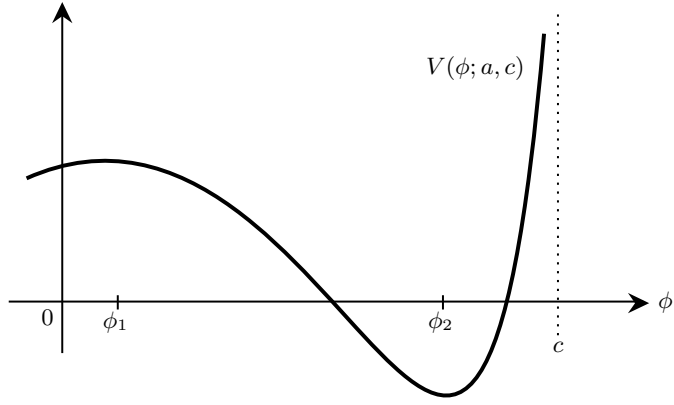
\begin{figure}[t!]
\begin{center} 
\tikzset{every picture/.style={line width=0.75pt}}
\begin{tikzpicture}[x=0.75pt,y=0.75pt,yscale=-1,xscale=1]
%Straight Lines [id:da09343087565797448] 
\draw    (177,195) -- (177,20.8) ;
\draw [shift={(177,17.8)}, rotate = 90] [fill={rgb, 255:red, 0; green, 0; blue, 0 }  ][line width=0.08]  [draw opacity=0] (10.72,-5.15) -- (0,0) -- (10.72,5.15) -- (7.12,0) -- cycle    ;
%Straight Lines [id:da660275341582703] 
\draw    (150,169) -- (470,169) ;
\draw [shift={(473,169.8)}, rotate = 180.14] [fill={rgb, 255:red, 0; green, 0; blue, 0 }  ][line width=0.08]  [draw opacity=0] (10.72,-5.15) -- (0,0) -- (10.72,5.15) -- (7.12,0) -- cycle    ;
%Straight Lines [id:da0532940217484017] 
\draw  [dash pattern={on 0.84pt off 2.51pt}]  (427,186) -- (427,23.8) ;
%Curve Lines [id:da6108107075274243] 
\draw [line width=1.5]    (159,106.8) .. controls (321,34.8) and (385,437.8) .. (420,33.8) ;
%Straight Lines [id:da5536719331283739] 
\draw    (205,173) -- (205,165.8) ;
%Straight Lines [id:da8082011272457064] 
\draw    (369,173) -- (369,165.8) ;
% Text Node
\draw (476,162.5) node [anchor=north west][inner sep=0.75pt]  [font=\footnotesize]  {$\phi$};
% Text Node
\draw (196,172.4) node [anchor=north west][inner sep=0.75pt]  [font=\footnotesize]  {$\phi _{1}$};
% Text Node
\draw (360,172.4) node [anchor=north west][inner sep=0.75pt]  [font=\footnotesize]  {$\phi _{2}$};
% Text Node
\draw (423,188) node [anchor=north west][inner sep=0.75pt]  [font=\footnotesize]  {$c$};
% Text Node
\draw (359,43.4) node [anchor=north west][inner sep=0.75pt]  [font=\footnotesize]  {$V(\phi;a,c)$};
% Text Node
\draw (165,171.4) node [anchor=north west][inner sep=0.75pt]  [font=\footnotesize]  {$0$};
\end{tikzpicture}
\vspace{-14em}
\caption{Depiction of the effective potential $V(\phi;a,c)$ for an admissible value of $a\in\R$. Note that there is a vertical asymptote at $\phi=c$, and that all the periodic solutions here exist for $\phi<c$, oscillating around the critical points $\phi_2$ of the potential $V$.}
\label{f:potential}
\end{center} 
\end{figure}
%%%%%%%%%%%%%%%%%%%%%%%%%%%%%%%%%%%%%%%%%%%%%%%%%%%%%%%%%%%%%%%%%%%%%%%%%%%%%%%%%%%%%%

By elementary phase plane analysis, the existence of smooth periodic solutions of \eqref{e:profile2} satisfying \eqref{e:condition} follows provided that the effective
potential function
\be
\label{e:potential}
V(\phi;a,c) := -\frac{1}{2}\phi^2 + \frac{a}{c-\phi}
\ee
has a local minimum for $\phi<c$.  The potential $V$ has been studied in several works (see \cite{EJ24,GMNP22}, for instance), where it is shown that for each $c>0$
and $a\in(0,4c^3/27)$ the potential has exactly two critical points $\phi_1,\phi_2\in(0,c)$.  More precisely, for such $a$ and $c$ the potential $V(\cdot;a,c)$
has a local maximum $\phi_1\in(0,c/3)$ and a local minimum $\phi_2\in(c/3,c)$ and, furthermore, satisfies
\[
V'(\phi;a,c)>0,~~{\rm for}~~\phi\in(0,\phi_1)\cup(\phi_2,c),~~~V'(\phi;a,c)<0~~{\rm for}~~\phi\in(\phi_1,\phi_2),
\]
and
\[
\lim_{\phi\to c^-}V(\phi;a,c)=+\infty.
\]
For a graphical representation, see Figure \ref{f:potential}.  
By the above phase-plane analysis it follows that if we define the set
\be
\label{e:manifold}
\mathcal{M} := \Big\{(a,E,c)\in\R^3 : c>0, \,\, a\in(0,4c^3/27), \,\, E\in \big(V(\phi_2(a,c);a,c),V(\phi_1(a,c);a,c)\big)\Big\},
\ee
then for $(a,E,c)\in\mathcal{M}$ the profile equation \eqref{e:profile} admits a one-parameter family, parameterized by translation invariance, of smooth periodic solutions $\phi(\eta;a,E,c)$ 
satisfying $\phi<c$ and with period
\be
\label{e:period}
T:=T(a,E,c) = \sqrt{2}\int_{\phi_{\min}}^{\phi_{\max}} \frac{d\phi}{\sqrt{E-V(\phi;a,c)}},
\ee
where here $\phi_{\min}$ and $\phi_{\max}$ denote, respectively the minimum and maximum roots of the equation $E-V(\cdot;a,c)=0$, and hence correspond 
to the minimum and maximum values of the corresponding periodic solution $\phi$.  
From the above analysis, it follows that the CH equation \eqref{e:ch} admits a four-parameter family, constituting a $C^1$ manifold, of periodic traveling wave solutions of the form
\be
\phi(x-ct+x_0;a,E,c), \quad x_0\in\R, \quad (a,E,c)\in\mathcal{M},
\ee
with period $T(a,E,c)$.

\begin{figure}[t!]
\begin{center}
\includegraphics[scale=0.6]{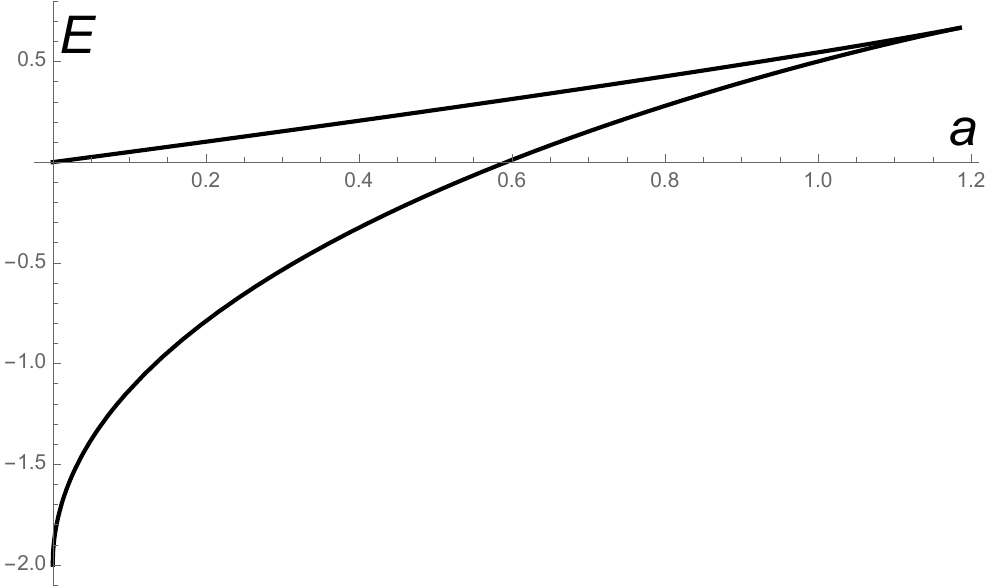}
\caption{A numerically generated plot of the intersection of the existence region $\mathcal{M}$ and the plane $c=2$.  
The upper and lower black curves correspond to solitary waves and constant solutions, respectively,
and act as the boundary of the existence region.  Note these upper and lower boundaries intersect at $(a,E) = (4c^3/27,c^2/6)$, while they intersect the left boundary
$a=0$ at $E=0$ and $E=-c^2/2$, respectively (recall here $c=2$). }\label{F:exist}
\end{center}
\end{figure}

In particular, note that since the values $\phi_{\min}$ and $\phi_{\max}$ are smooth functions of the traveling wave parameters 
$(a,E,c)\in\mathcal{M}$, it follows that\footnote{Indeed, the integral representation \eqref{e:period} can be regularized by a standard procedure at the
square root branch points, thus yielding the desired smoothness.} the period function $T(a,E,c) \in C^1(\mathcal{M})$.  Of particular interest in our work, as we will
see, is the monotonicity properties of the period $T$ on the parameter $a$.  For later reference, we record the following result.

\begin{lemma}[Period Monotonicity]\label{L:period}
The period function $\mathcal{M}\ni(a,E,c)\to T(a,E,c)$ satisfies the following:
\begin{itemize}
\item[(i)] $T_a>0$ provided that $E\in\left(-\frac{1}{2}c^2,-\left(1-\sqrt{\frac{2}{3}}\right)c^2\right)$.
\item[(ii)] $T_a<0$ provided that $E\in\left(0,\frac{1}{6}c^2\right)$.
\item[(iii)] $T_a$ has a unique local maximum in $a$ if $E\in\left(-\left(1-\sqrt{\frac{2}{3}}\right)c^2,0\right)$.
\end{itemize}
\end{lemma}

For details, see \cite[Theorem 2.5]{GV15} as well as \cite[Remark 2.8]{GMNP22}.  In particular, we note that by the analysis in \cite[Remark 2.5]{GMNP22}
the range $E\in\left(-\frac{1}{2}c^2,\frac{1}{6}c^2\right)$ constitutes the entire admissible range of the parameter $E$ for $(a,E,c)\in\mathcal{M}$.  See Figure \ref{F:exist}
above, as well as Figure 2.3 in \cite{GMNP22}.  It follows that outside of a co-dimension one surface in $\mathcal{M}\subset\RM^3$ the period function
satisfies $T_a(a,E,c)\neq 0$.

\subsection{Hamiltonian Structure and Conserved Qunatities}

Since we are interested in the local dynamics of the CH equation about periodic tracveling waves, throughout we consider the conserved quantities and Hamiltonian formulation
as being posed on appropriate subspaces of $L^2_{\rm per}(0,T)$ for some period $T>0$.  Being a completely integrable PDE, the CH equation admits infinitely many conserved quantities.
As we will see, however, our work only makes use of the following three conserved quantities for $T$-periodic solutions:
\bse
\label{e:conserved}
\begin{align}
\label{e:mass}
M(u) &= \int_0^T u\,dx, \\
\label{e:momentum}
P(u) &= \frac{1}{2}\int_0^T \left(u^2 + u_x^2\right)\,dx, \\
\label{e:energy}
F(u) &= \frac{1}{2}\int_0^T\left(u^3 + uu_x^2\right),
\end{align}
\ese 
representing the mass, momentum, and energy, respectively.  Similar to our consideration of the period function $T(a,E,c)$ in \eqref{e:period}, the 
conserved quantities above can be restricted to the manifold $\mathcal{M}$ of periodic traveling wave solutions of \eqref{e:ch}.  Indeed,
given a $T(a,E,c)$-periodic traveling wave $\phi(\cdot;a,E,c)$, we can define the functions $M,P:\mathcal{M}\to\RM$ via\footnote{The conserved quantity $F$ can of course also
be restricted to $\mathcal{M}$ to yield a $C^1$ function on $\mathcal{M}$.  Since this is not used in our analysis, however, we omit writing out the formulas.}
\begin{align*}
M(a,E,c)&=\int_0^{T(a,E,c)}\phi(\eta;a,E,c)d\eta = \sqrt{2}\int_{\phi_{\rm min}}^{\phi_{\rm max}}\frac{\phi~d\phi}{\sqrt{E-V(\phi;a,c)}}\\
%P(a,E,c)&=\frac{1}{2}\int_0^{T(a,E,c)}\phi(\eta;a,E,c)d\eta = \frac{\sqrt{2}}{2}\int_{\phi_{\rm min}}^{\phi_{\rm max}}\frac{\phi^2~d\phi}{\sqrt{E-V(\phi;a,c)}}
P(a,E,c)&=\frac{1}{2}\int_0^{T(a,E,c)}\phi(\eta;a,E,c)d\eta = \frac{\sqrt{2}}{2}\int_{\phi_{\rm min}}^{\phi_{\rm max}}\left(\frac{\phi^2}{\sqrt{E-V(\phi;a,c)}}+\sqrt{E-V(\phi;a,c)}\right)~d\phi.
\end{align*}
As before, the above integrals can be regularlized near the square root singularities and hence represent $C^1$ functions on $\mathcal{M}$.

\begin{remark}\label{r:mass}
For a  given periodic traveling wave solution $\phi(\cdot;a,E,c)$ of the CH equation, we note from \eqref{e:profile2} that
\[
M=\int_0^T\frac{a}{(c-\phi)^2}dx
\]
is necessarily positive thanks to \eqref{e:manifold}.  That is, $M(a,E,c)>0$ for all $(a,E,c)\in\mathcal{M}$.  
\end{remark}

Regarding the energy functional $F$, we note that the profile equation \eqref{e:profile} is precisely the Euler-Lagrange equation for the action functional
\be
\Lambda(u) = cP(u) - F(u) - EM(u),
\ee
which we can use to write \eqref{e:ch} in traveling cooridinates in the Hamiltonian form
\be
u_t = \mathcal{J}\frac{\delta\Lambda}{\delta u}(u),
\ee
where  here
\be
\mathcal{J} := (1-\d_x^2)^{-1}\d_x = \d_x(1-\d_x^2)^{-1}
\ee
is a skew-adjoint operator on $L^2_{\rm per}(0,T)$.  In our forthcoming stability analysis, we will be interested in the linearization of \eqref{e:ch} about
a given $T$-periodic traveling wave solution $\phi$.  As is common in Hamiltonian stability problems, the stability of the background wave $\phi$
relies strongly on the spectral properties of the (symmetric) Hessian operator
\begin{equation}\label{e:L}
\mathcal{L}[\phi]:=\frac{\delta^2\Lambda}{\delta u^2}(\phi)=-\partial_\eta\left(c-\phi\right)\partial_\eta  + \left(c-3\phi+\phi''\right)
\end{equation}
considered as acting on $L^2_{\rm per}(0,T)$.  Of particular interest is the dimension of the kernel of $\mathcal{L}$: as we will see, it is always at least one-dimensional
due to the invariance of \eqref{e:ch} with respect to spatial translations, and in general one hopes this accounts for everything in the kernel. 
The following result, established in \cite{GMNP22},  connects the dimension of the $T$-periodic kernel of $\mathcal{L}$ to monoconicity properties of the
period function $T$ in \eqref{e:period}.

\begin{proposition}\label{P:kernelL}
For a given $(a,E,c)\in\mathcal{M}$, let $\phi(\cdot;a,E,c)$ be the associated periodic traveling wave solution of \eqref{e:ch} with period $T(a,E,c)$.
If $T_a\neq 0$, then the kernel of the Hessian operator $\mathcal{L}[\phi]:H^2_{\rm per}(0,T)\subset L^2_{\rm per}(0,T)\to L^2_{\rm per}(0,T)$ satisfies
\begin{equation}\label{e:kernelL}
{\rm ker}\left(\mathcal{L}[\phi]\right)={\rm span}\left\{\phi'\right\}.
\end{equation}
\end{proposition}

For details of the proof, which relies on Floquet theory as well as Sylvester's inertial law theorem, see \cite[Theorem 3.6]{GMNP22}.  Note
that by Lemma \ref{L:period} above, it follows that the kernel of $\mathcal{L}[\phi]$ is one-dimensional outside of a co-dimension 
one surface in $\mathcal{M}\subset\RM^3$.  In particular, the equality \eqref{e:kernelL} holds for almost every periodic
traveling wave solution of the CH equation \eqref{e:ch}.  Throughout our upcoming work, we will assume that $T_a\neq 0$
and hence that the kernel of $\mathcal{L}[\phi]$ is generated only through translation invariance.

\begin{remark}
The anlaysis in \cite{GMNP22} further shows that  the kernel of $\mathcal{L}[\phi]$ is two-dimensional whenever $T_a=0$.  This additional
direction is not, however, associated to a continuous Lie symmetry of the underlying equation.  As such, we will not consider this co-dimension one (hence degenerate)
case in our work.
\end{remark}

%
%Finally, a number of Jacobian determinants will arise throughout the work. For notational simplicity, we adopt the following notation for $2\times 2$ Jacobians
%%
%\be
%\{f,g\}_{x,y} := \det\left(\frac{\d(f,g)}{\d(x,y)}\right) = \det\begin{pmatrix} f_x & f_y \\ g_x & g_y \end{pmatrix},
%\ee
%%
%and similarly $\{f,g,h\}_{x,y,z}$ for the analogous $3\times 3$ Jacobian. 

Before continuing, we briefly note that the Camassa-Holm equation admits two additional Hamiltonian structures, both of which are best expressed
in terms of the momentum density $m=u-u_{xx}$ of solutions of \eqref{e:ch}.  See, for example, \cite{EJ24,GMNP22,LP22}, where the authors use
these additional Hamiltonian structures to study the nonlinear orbital stability of periodic and solitary wave solutions of the CH equation.  In our work,
we focus exclusively on the Hamiltonian structure above, expressed locally in terms of the unknown $u$, since the associated conserved quantities
naturally arise from Whitham's  theory of modualtions.

\subsection{On Parameterizations of Periodic Traveling Waves}\label{s:param}

The goal of this work is to connect predictions from Whitham's theory of wave modulations to the rigorous spectral stability of periodic traveling waves.  A very interesting
but subtle point here is that often these theories use different parameterizations of the manifold of periodic traveling waves.  We have already
seen in Section \ref{s:existence} that for the Camassa-Holm equation the set of periodic traveling waves can be found by integrating to quadrature, thus giving
a parameterization of the manifold of periodic traveling waves by the coordinates $(a,E,c)\in\mathcal{M}$: see \eqref{e:manifold}.  

In Whitham's theory, however, we will see in the next section that   the slow modulational dynamics are described in terms of a first order system of 
conservation laws that encode the slow-time evolution of the frequency $k$, defined as one over the period $T$ of the wave, and 
the conserved quantities $M$ and $P$, from \eqref{e:mass}-\eqref{e:momentum} above, evaluated along the slowly varying wave: see \eqref{e:Wsystem} in 
Section \ref{s:W} below.
That is, a fundamental (albeit oftentimes implicit) tenet in Whitham's theory is that the manifold of periodic traveling waves can be locally parameterized by 
the frequency $k$ and the conserved quantities $M$ and $P$ of the wave.  

This, of course, poses an interesting problem if ones goal is to compare the mathematical stability theory, where the manifold of periodic traveling wave solutions is parametrized by
the traveling wave parameters $(a,E,c)$, and Whitham's theory of modulations, where the manifold of such solutions is parameterized 
by $(k,M,P)$.  As such, when doing such comparisons it is natural to assume that the nonlinear mapping
\[
\mathcal{M}\ni(a,E,c)\mapsto (k,M,P)\in\RM^3
\]
is locally invertible near the periodic traveling waves under consideration.  By the Implicit Function Theorem, it is sufficient to assume that the Jacobian of this mapping 
is non-zero, i.e. that
\begin{equation}\label{e:invert}
\det\left(\frac{\partial(k,M,P)}{\partial(a,E,c)}(\phi)\right)\neq 0
\end{equation}
for periodic traveling wave solutions $\phi$ of \eqref{e:ch}.  In effect, it is only possible to compare these two theories for periodic traveling 
waves $\phi$ for which \eqref{e:invert} holds.  As such, throughout our work we will at all times assume that \eqref{e:invert} holds and will only make
this assumption explicit when necessary.

\section{The Whitham Modulation Equations}\label{s:W}

In this section, we begin our study of the large-space/long-time dynamics of an arbitrary amplitude, slowly modulated periodic traveling wave solution of the CH equation \eqref{e:ch}. 
A formal approach to study the dynamical behavior of such slowly modulated periodic traveling waves is to analyze the 
associated Whitham modulation equations \cite{Whitham}. While the technique was origninally developed using averaged conservation laws it was 
later shown to be equivalent to an asymptotic reduction derived through formal multiple-scales (WKB) expansions (see \cite{L66}).   For the CH equation \eqref{e:ch},
the Whitham modulation equations were derived in \cite{AG05} through the method of averaged Lagrangians.  For our work, however, we believe it is more illustrative
to do this derivation through the use of WKB methods.  Indeed, when using multiple scales the Whitham modulation equations arise through the application of the Fredholm
alternative\footnote{And also Clairaut's Theorem, which yields conservation of waves.} which, at its core, boils down to studying projections of the asymptotic problem onto
appropriate finite-dimensional subspaces.  This is analogous to the upcoming rigorous theory, which utilizes spectral perturbation theory to project the infinite-dimensional spectral
problem onto a finite-dimensional subspace.  As such, for completeness (and to push our own vision of this program), we choose here to provide a self-contained
derivation of the Whitham modulation equations for the CH equation \eqref{e:ch}.

\subsection{Whitham Theory for CH}

Since modulations of periodic waves occur on very large space and time scales, when compared to the local spatial and temporal scales of the wave, we
begin by introducing  ``slow'' variables $(X,T):=(\eps x,\eps t)$ and notice that, in the slow coordinates, \eqref{e:ch} becomes 
\be
\label{e:slowscale}
u_T - \eps^2u_{XXT} + 3uu_X = 2\eps^2u_Xu_{XX} + \eps^2uu_{XXX}.
\ee
To understand how periodic traveling wave solutions modulate slowly in space and time, we follow Whitham \cite{W65} and seek a multiple-scales solution to \eqref{e:slowscale} 
of the form
\be
\label{e:ansatz}
u(X,T;\eps) = u^0(X,T,\theta) + \eps u^1(X,T,\theta) + \mathcal{O}(\eps^2),
\ee
where here
\be
\label{e:theta}
\theta:= \eps^{-1}\psi(X,T),
\ee
and the phase $\psi$ is chosen so that the functions $u^j$ appearing in the expansion are $1$-periodic functions of $\theta$.  

%
%{\color{red}
%Notice it is necessary to fix the wave period in order to avoid secularity. Moreover, we identify  $\theta_x=\psi_X=k(X,T)$, and $\theta_t=\psi_T=-\omega(X,T)$ as the spatial and temporal frequencies of the modulation, respectively. The parameter $\eps$ characterizes the ratio of the wave's typical wavelength to the typical modulation length scale determined by initial and or bounbdary conditions.
%}

Next, substituting the expansion \eqref{e:ansatz} into \eqref{e:slowscale} and collecting like powers of $\eps$ yields a hierarchy of equations in algebraic orders of $\eps$ that must be simultaneuously satisfied. At $\mathcal{O}(1)$, we obtain
\[
 \psi_T u^0_{\theta} -\psi_T\psi_{X}^2 u^0_{\theta\theta\theta} + 3\psi_X u^0u^0_{\theta} = 2\psi_X^3u^0_{\theta}u^0_{\theta\theta} + \psi_X^3u^0u^0_{\theta\theta\theta}.
\]
Setting $k:=\psi_X$ and $\omega := -\psi_T$ as the spatial and temporal frequencies of the modulation, respectively, and $c:=\omega/k$ as the wave speed, the above
is immediately recognized as the profile equation \eqref{e:profile}, rescaled for $1$-periodic functions.   
Thus, for fixed $X$ and $T$ we may choose $u^0$ to be a periodic traveling wave solution of \eqref{e:ch}, and hence we have
\be
\label{e:u0}
u^0(\theta;X,T) = \phi(\theta;a(X,T), E(X,T), c(X,T)),
\ee 
for some even\footnote{Note the profile equation is invariant with respect to spatial translations and the discrete symmetry $\theta\mapsto -\theta$.} solution  $\phi$ of \eqref{e:quad} with 
$(a(X,T),E(X,T),c(X,T))\in\mathcal{M}$. Note here the traveling wave parameters $a$, $E$ and $c$ are functions of the slow variables $X$ and $T$, i.e. they are themselves
slowly varying functions of space and time.  Further, the consistency condition $\psi_{XT}=\psi_{TX}$ implies
\begin{equation}\label{e:conswaves}
k_T+\left(kc\right)_X=0,
\end{equation}
a sort of (slowly varying) nonlinear dispersion relation that is often referred to as conservation of waves.

Continuing, the $\mathcal{O}(\eps)$ equation reads
\begin{equation}\label{e:O(e)}
%&\mathcal{O}(\eps): \mathcal{W}[u^0]u^1 = \mathcal{G}u^0_T + F(u^0)\\
 -k\partial_\theta\mathcal{L}[u^0]u^1 = \mathcal{G}u^0_T + F(u^0)
\end{equation}
where here $\mathcal{L}[u^0]$ is the operator defined in \eqref{e:L} rescaled for $1$-periodic functions,
\[
\mathcal{G} := 1-k^2\d_{\theta}^2,
\]
and
\be
\label{e:F}
\begin{multlined}
F(u^0) := -2k\d_X(\omega u^0_{\theta\theta}) 
- k_X\omega u^0_{\theta\theta} - 3u^0u^0_X + 4k^2u^0_{\theta}u^0_{X\theta} \\
\hspace{18mm} + 2kk_X(u^0_{\theta})^2 + 2k^2u^0_Xu^0_{\theta\theta} + 3k^2u^0u^0_{\theta\theta X} + 3kk_Xu^0u^0_{\theta\theta}.
\end{multlined} 
\ee
In order to maintain the uniform asymptotic ordering of the ansatz \eqref{e:ansatz} we seek a solution $u^1$ to \eqref{e:O(e)}
that is bounded and $1$-periodic in $\theta$.  
Since \eqref{e:O(e)} is a forced linear equation for the unknown $u^1$, and observing that the adjoint of $\partial_\theta\mathcal{L}[u^0]$ on $L^2_{\rm per}(0,1)$
is
\[
\left(\partial_\theta\mathcal{L}[u^0]\right)^\dag = -\mathcal{L}[u^0]\partial_\theta,
\]
it follows from the Fredholm alternative that \eqref{e:O(e)} is solvable in the space of $1$-periodic functions if and only if
\begin{equation}\label{e:Fredholm}
%\mathcal{G}u^0_T + F(u^0) \perp \ker\left(\mathcal{W}^{\dagger}[u^0]\right),
\mathcal{G}u^0_T + F(u^0) \perp \ker\left(\mathcal{L}[u^0]\partial_\theta\right).
\end{equation}
%where
%\[
%%\mathcal{W}^{\dagger}[u^0] := -k^2\omega\d_{\theta}^3 + k^3u^0\d_{\theta}^3 + k^3(u^0_{\theta\theta}\d_{\theta} + u^0_{\theta}\d_{\theta}^2) + \omega\d_{\theta} - 3ku^0\d_{\theta}.
%\mathcal{W}^{\dagger}[u^0] := -k\mathcal{L}[u^0]\partial_\theta
%\]
%denotes the adjoint of $\mathcal{W}[u^0]$ on $L^2_{\rm per}(0,1)$.  
Under the assumption that the hypothesis of Proposition \ref{P:kernelL} holds,\footnote{Recall the hypothesis
of Proposition \ref{P:kernelL}, namely that $T_a\neq 0$, holds generically.  In fact, due to the work in \cite{GMNP22} it holds off of a co-dimension one set in $\mathcal{M}$.}
it follows that 
\[
%\ker\left(\mathcal{W}^{\dagger}[u^0]\right)={\rm span}\left\{1,u^0\right\}
\ker\left(\mathcal{L}[u^0]\partial_\theta\right)={\rm span}\left\{1,u^0\right\}
\]
and thus \eqref{e:Fredholm} is equivalent to the orthogonality conditions
\begin{equation}\label{e:solv1}
 \left< 1, \mathcal{G}u^0_T + F(u^0)\right>=0 ~~~{\rm and}~~~\left< u^0, \mathcal{G}u^0_T + F(u^0)\right>=0.
\end{equation}
Note for brevity here and throughout our work we take $\left<\cdot,\cdot\right>:= \left<\cdot,\cdot\right>_{L^2_{\rm per}(0,1)}$. 

We now aim to express the solvability conditions \eqref{e:solv1} in a more convenient  form.  To this end, we observe that
\[
\left< 1, 4k^2u^0_{\theta}u^0_{X\theta} + 2kk_X(u^0_{\theta})^2 + 2k^2u^0_Xu^0_{\theta\theta} \right>= \left< 1, 2k^2(u^0_Xu^0_{\theta})_{\theta} + (k^2(u^0_{\theta})^2)_X\right>,
\]
and that integration by parts yields 
$$
\left< 3k^2u^0u^0_{\theta\theta X} + 3kk_Xu^0u^0_{\theta\theta} \right> = \left< 1, -\thalf(k^2(u^0_{\theta})^2)_X \right>,
$$
it follows from the explicit forms of $F$ and $\mathcal{G}$ above the first solvability condition above is equivalent to
\[
M_T = \left< 1,-\thalf(u^0)^2 - \half k^2(u^0_{\theta})^2 \right>_X,
\]
where $M(u^0)=\int_0^1u^0\,d\theta$ is the conserved quantity $M$, the mass in \eqref{e:mass} evaluated at the $1$-periodic traveling wave $u^0$. 
%Also, for brevity here and throughout the work $\left<\cdot,\cdot\right>:= \left<\cdot,\cdot\right>_{L^2_{\rm per}(0,1)}$. 
Similarly, using conservation of waves \eqref{e:conswaves} and integration by parts we find the identities
\begin{align*}
&\left< u^0, -2k\omega_Xu^0u^0_{\theta\theta} - 2k\omega u^0u^0_{\theta\theta X} - k_X\omega u^0u^0_{\theta\theta} \right> = \left< 1, -\half(k^2)_T(u^0_{\theta})^2 + (k^2\omega (u^0_{\theta})^2)_X \right>\\
&\left< u^0, 2k^2u^0_Xu^0_{\theta\theta} + 3k^2u^0u^0_{\theta\theta X} + 3kk_Xu^0u^0_{\theta\theta} \right> = \left< 1, (k^2(u^0)^2u^0_{\theta\theta})_X - 2k^2((u^0)^2)_{\theta}u^0_{\theta X} + kk_X(u^0)^2u^0_{\theta\theta}\right>, 
\end{align*} 
and
$$
\left< u^0, -u^0_T + k^2u^0_{T\theta\theta} \right> = \left< 1, -\half(((u^0)^2)_T + k^2((u^0_{\theta})^2)_T) \right>.
$$
Thus, again using the explicit forms of $F$ and $\mathcal{G}$ above, the second solvability condition in \eqref{e:solv1} is equivalent to
\[
P_T = \left< 1, k\omega(u^0_{\theta})^2 - (u^0)^3 + k^2(u^0)^2u^0_{\theta\theta}\right>_X,
\]
where  here $P(u^0)=\half\int_0^1 ((u^0)^2 + k^2(u^0_{\theta})^2)\,d\theta$ is the conserved quantity $P$, the momentum in \eqref{e:momentum}, evaluated at the $1$-periodic traveling wave $u^0$.

Taken all together, the consistency condition \eqref{e:conswaves} and the  solvability conditions \eqref{e:solv1} yield the  system of equations
\be
\label{e:Wsystem}
\begin{cases}
k_T + (kc)_X = 0\\
M_T = \left< 1, -\frac{3}{2}(u^0)^2 - \frac{1}{2}k^2(u^0_{\theta})^2\right>_X \\
P_T = \left< 1, ck^2(u^0_{\theta})^2 - (u^0)^3 + k^2(u^0)^2u^0_{\theta\theta}\right>_X \\
\end{cases}
\ee
which, is referred to as the Whitham modulation equations for the Camassa-Holm equation \eqref{e:ch}.   The system  \eqref{e:Wsystem} is, to leading order,
expected to govern the  evolution of the wave number $k$ and the conserved quantities $M$ and $P$ of a slow modulation of the periodic traveling wave $u^0$.
We now wish to understand the predictions of \eqref{e:Wsystem} applied to a given periodic traveling wave solution of \eqref{e:ch}.  To this end,
suppose that $(a_0,E_0,c_0)\in\mathcal{M}$ and let $\phi(\cdot;a_0,E_0,c_0)$ be the corresponding even, $T=1/k$-periodic solution of \eqref{e:quad}. 
The above multiple-scales analysis suggests that, for at least some short time,\footnote{The length of time this formal solution remains relates to the 
the length of time the slowly evolving functions $(a(\eps x,\eps t), E(\eps x,\eps t), c(\eps x,\eps t))$ remain near $(a_0,E_0,c_0)$.}
the Camassa-Holm equation admits a slowly modulated periodic traveling wave solution of the form
\be\label{e:approx}
u(x,t;\eps) = \phi(\eps^{-1}\psi(\eps x,\eps t);a(\eps x,\eps t), E(\eps x,\eps t), c(\eps x,\eps t)) + O(\eps), \quad \eps\to 0^+,
\ee
where the parameters $(a(\eps x,\eps t), E(\eps x,\eps t), c(\eps x,\eps t))\in\mathcal{M}$ are slowly evolving near $(a_0,E_0,c_0)$ 
in such a way that $k$, $M$, and $P$ satisfy the Whitham modulation system \eqref{e:Wsystem}.  Note that, importantly,
this formal description requires (or presumes) that the slow evolution of $(k,M,P)$ described in \eqref{e:Wsystem} is equivalent to a slow
evolution of the traveling wave parameters $(a,E,c)$, i.e. it requires that the map
\[
\mathcal{M} \ni (a,E,c) \mapsto (k,M,P) \in \R^3 
\]
be locally invertible near $(a_0,E_0,c_0)$.   By the Implicit Funciton Theorem, this is of course guaranteed by requiring that the Jacobian
of this mapping, given by \eqref{e:invert} in Section \ref{s:param} above, is non-zero at the point $(a_0,E_0,c_0)$.  Note that
condition \eqref{e:invert} is effectively equivalent to the Whitham system \eqref{e:Wsystem} describing an evolutionary system on the manifold $\mathcal{M}$
of periodic traveling wave solutions, as constructed in Section \ref{s:existence} above.

Under the non-degeneracy assumption \eqref{e:invert}, the stability of $\phi(a_0,E_0,c_0)$ to slow modulations is therefore expected to be 
determined by the linearization of the Whitham modulation system \eqref{e:Wsystem} about $\phi$.  To this end, note that using the chain rule we can rewrite \eqref{e:Wsystem}
as the quasilinear system
\be
\label{e:Wsys} 
\begin{pmatrix} k \\ M \\ P \end{pmatrix}_T = \mathcal{W}(u^0)\begin{pmatrix} k \\ M \\ P \end{pmatrix}_X,
\ee
where 
\begin{equation}\label{e:D_Whitham}
\mathcal{W}(u) = \begin{pmatrix} -kc_k-c & -kc_M & -kc_P \\ \langle 1,-\frac{3}{2}u^2-\frac{1}{2}k^2u^2_{\theta}\rangle_k & \langle 1,-\frac{3}{2}u^2-\frac{1}{2}k^2u^2_{\theta}\rangle_M & \langle 1,-\frac{3}{2}u^2-\frac{1}{2}k^2u^2_{\theta}\rangle_P \\ \langle 1,ck^2u_{\theta}^2 - u^3 + k^2u^2u_{\theta\theta}\rangle_k & \langle 1,ck^2u_{\theta}^2 - u^3 + k^2u^2u_{\theta\theta}\rangle_M & \langle 1,ck^2u_{\theta}^2 - u^3 + k^2u^2u_{\theta\theta}\rangle_P \end{pmatrix}.
\end{equation}
Noting that the periodic traveling wave $\phi$ is a constant solution with respect to the slow evolutionary variables $(X,T)$, 
it is natural to expect that the stability of $\phi$ to slow modulations is determined by the eigenvalues of the matrix $\mathcal{W}(\phi)$.  More specifically, note
that linearizing \eqref{e:Wsys} about the periodic traveling wave $\phi(\cdot;a_0,E_0,c_0)$, treated here as an equilibrium solution of \eqref{e:Wsys},
yields a linear evolution equation for $(k,M,P)$ whose eigenvalues will be of the form
\[
\widetilde{\lambda}_j(\xi) = i\xi\alpha_j
\]
where here the $\{\alpha_j\}_{j=1}^3$ are the eigenvalues of the matrix $\mathcal{W}(\phi)$.  In particular, if the eigenvalues of $\mathcal{W}(\phi)$ are all real, corresponding
to the Whitham system \eqref{e:Wsys} being weakly hyperbolic, then one should expect at least marginal (spectral) stability of the periodic traveling wave $\phi$
with respect to slow modulations.  If, on the other hand, $\mathcal{W}(\phi)$ has an eigenvalue with non-zero imaginary part, corresponding to the system \eqref{e:Wsystem} being elliptic
at $\phi$, then one expects $\phi$ to be (spectrally) unstable to slow modulations.

%
%
%Recall, $u^0$ depends on $a$, $E$, and $c$, all of which depend of $k$, $M$, and $P$, which in turn all depend on $X$ and $T$. So, we may use the chain rule to expand the right hand side as a quasi-linear system
%%
%\be
%\label{e:Wsys} 
%\begin{pmatrix} k \\ M \\ P \end{pmatrix}_T = D(u^0)\begin{pmatrix} k \\ M \\ P \end{pmatrix}_X,
%\ee
%%
%where
%%
%$$
%D(u) = \begin{pmatrix} -kc_k-c & -kc_M & -kc_P \\ \langle 1,-\frac{3}{2}u^2-\frac{1}{2}k^2u^2_{\theta}\rangle_k & \langle 1,-\frac{3}{2}u^2-\frac{1}{2}k^2u^2_{\theta}\rangle_M & \langle 1,-\frac{3}{2}u^2-\frac{1}{2}k^2u^2_{\theta}\rangle_P \\ \langle 1,ck^2u_{\theta}^2 - u^3 + k^2u^2u_{\theta\theta}\rangle_k & \langle 1,ck^2u_{\theta}^2 - u^3 + k^2u^2u_{\theta\theta}\rangle_M & \langle 1,ck^2u_{\theta}^2 - u^3 + k^2u^2u_{\theta\theta}\rangle_P \end{pmatrix}.
%$$

As mentioned in the introduction, the goal of this work is to rigorously justify the above predictions.  In particular, we aim to prove that a necessary condition
for a given periodic traveling wave solution $\phi_0(\cdot;a_0,E_0,c_0)$ to be stable with respect to slow modulations is the weak hyperbolicity
of the Whitham modulation system \eqref{e:Wsystem} at the point $(k_0,M(\phi_0),P(\phi_0))$.  This is the goal of the next section.

\subsection{Modulations of Stokes Waves for the CH Equation}\label{s:W_stokes}

In this section, we investigate the Whitham modulation equations in \eqref{e:Wsystem} in case of periodic traveling waves
with asymptotically small oscillations about their mean.  By the work of Abenda and Grava in \cite{AG05}, we expect to see
that the associated Whitham system (derived carefully below) will always be hyperbolic.

The Stokes waves for the CH equation are necessarily equilibrium solutions of the rescaled evolutionary equation
\eqref{e:stationary_rescale}, i.e. they are $1$-periodic solutions of the rescaled profile equation
\begin{equation}\label{e:profile_rescale2}
-\omega\left(\phi'-k^2\phi'''\right)+3k\phi\phi'=k^3\left(2\phi'\phi''+\phi\phi'''\right),
\end{equation}
where here primes denote differentiation with respect to the traveling variable $\theta=kx-\omega t$ and $\omega=kc$.  Using a straightforward Lyapunov-Schmidt
argument, one can prove that solution pairs $(\phi,\omega)$ of \eqref{e:profile_rescale2} with asymptotically small oscillations about its mean $M$ admit a convergent
expansion in $H^3_{\rm per}(0,1)$ of the form
\begin{equation}\label{e:small_expand}
\left\{\begin{aligned}
\phi(\theta;k,M,A)&=M+A\cos(2\pi\theta)+A^2\frac{(1+4k^2\pi^2)^2}{32k^2 M\pi^2}\cos(4\pi\theta)+\mathcal{O}(A)^3\\
\omega(k,M,A)&=\omega_0(k,M)+A^2\omega_2(k,M)+\mathcal{O}(A^4)
\end{aligned}\right.
\end{equation}
where $|A|\ll 1$ and
\[
\omega_0(k,M):=\frac{3kM+4k^3M\pi^2}{1+4k^2\pi^2}~~{\rm and}~~
 \omega_2(k,M):=\frac{3\left(1+4k^2\pi^2\right)^2}{64kM\pi^2}.
\]
For details, see Appendix \ref{a:stokes}.  

Substituting these expansions into the first two equations in  the Whitham modulation system \eqref{e:Wsystem} yields
\[
\left\{\begin{aligned}
&k_T+\partial_X\left(\omega_0+A^2\omega_2\right)=\mathcal{O}(A^3)\\
&M_T=\partial_X\left(-\frac{3}{2}M^2-\frac{3+4k^2\pi^2}{4}A^2\right)+\mathcal{O}(A^3),
\end{aligned}\right.
\]
while substituting the expansions into the momentum equation yields
\[
\left(\frac{1}{2}M^2+\frac{1+4k^2\pi^2}{4}A^2\right)_T = \left(-M^3-\frac{M}{2}\left(1+4k^2\pi^2+\frac{2}{1+4k^2\pi^2}\right)A^2\right)_X+\mathcal{O}(A^3).
\]
Using the chain rule, the above can be rewritten as the quasi-linear system
\[
%\left(\begin{array}{c}k\\ M\\ A\end{array}\right)_T+\mathcal{B}(k,M,A)\left(\begin{array}{c}k\\ M\\ A\end{array}\right)_X=0
\left(\begin{array}{c}k\\ M\\ A\end{array}\right)_T=B(k,M,A)\left(\begin{array}{c}k\\ M\\ A\end{array}\right)_X
\]
where the matrix function $B(k,M,A)$ expands as
\[
B(k,M,A)=-B_0(k,M)-B_1(k,M)A+\widetilde{B}(k,M,A)A^2
\]
with $\widetilde{B}(k,M,A)$ a bounded matrix-valued function and
\[
B_0=\left(\begin{array}{ccc}
				\frac{\partial\omega_0}{\partial k} & \frac{\partial\omega_0}{\partial M} & 0\\
				0 & 3M & 0\\
				0 & 0 & \frac{\partial\omega_0}{\partial k}
				\end{array}\right),
~~~~
B_1=\left(\begin{array}{ccc}
				0 & 0 & 2\omega_2\\
				0 & 0 & \frac{3+4k^2\pi^2}{2}\\
				\frac{1}{2}\frac{\partial^2\omega_0}{\partial k^2} & \frac{3-4k^2\pi^2}{(1+4k^2\pi^2)^2} & 0
%				\frac{1}{2}\frac{\partial^2\omega_0}{\partial k^2} & -16kM\pi^2\frac{\partial^2\omega_0}{\partial k^2} & 0
				\end{array}\right).
\]
A straightforward calculation shows eigenvalues of the linearized Whitham matrix $B(k,M,A)$ expand as
\[
\left\{\begin{aligned}
\lambda_1(k,M,A)&=-3M+\mathcal{O}(A^2)\\
\lambda_{\pm}(k,M,A)&=-\frac{\partial\omega_0}{\partial k} \pm \left(\frac{3-4k^2\pi^2}{4(1+4k^2\pi^2)}\right)A+\mathcal{O}(A^2).
\end{aligned}\right.
\]
In particular, it follows that for the Stokes waves for the Camassa-Holm equation that the associated Whitham system is always hyperbolic and, in fact, 
is strictly hyperbolic for $k\neq \sqrt{3}/2\pi$.  

According to Theorem \ref{T:main} it follows the the Stokes waves for the CH equation with $k\neq\sqrt{3}/2\pi$ are all spectrally stable in a negiborhood of the origin, i.e., such 
waves are all (spectrally) modulationally stable.  The lack of ellipticity for the Whitham system in this case is consistent with the work of Abenda and Grava in \cite{AG05}.
Further, we see that strict hyperbolicity in this Stokes wave limit holds generically.

\section{Rigorous Modulational Stability Theory}\label{s:rigorous}

In this section, we conduct a mathematically rigorous analysis of the spectral problem associated with the linearization of \eqref{e:ch} about periodic traveling wave solutions,
as constructed in Section \ref{s:existence} above, when subject to localized (i.e., integrable) perturbations on the line.  
As a first step, we discuss the required functional analytic setup regarding the linearized operators and the description of their spectrum
via the Floquet-Bloch theory.  We then study the so-called unmodulated operators, which effectively corresponds to studying the generalized kernel
of the linearization of \eqref{e:ch} about a given $T$-periodic solution on the space $L^2_{\rm per}(0,T)$.  Using this as a starting point, we then use
spectral perturbation theory to obtain an asymptotic description of the spectrum of the linearization considered as an operator on $L^2(\RM)$ in a sufficiently
small neighborhood of the origin.

\subsection{Linearization}
\label{s:physical}

To begin our work,  let $\phi$ be a $T$-periodic traveling wave solution of \eqref{e:ch}. Building off of the discussion in Section \ref{s:W} right after equation \eqref{e:approx}, we will
assume throughout this section that the non-degeneracy condition \eqref{e:invert} holds at $\phi$, and hence that the manifold $\mathcal{M}$ of periodic traveling 
wave solutions can be reparameterized in terms of the frequency $k=1/T$  and the conserved quantities $M$ and $P$.  In particular,
it follows that we can consider $\phi$ as being smoothly parameterized as $\phi(\cdot;k,M,P)$.  To emphasize the dependence on the frequency $k$, we rescale
the spatial variable as $y=kx$, so that $T=1/k$ periodic traveling wave solutions of \eqref{e:ch} correspond to $1$-periodic traveling wave solutions
of the rescaled evolution equation 
\be
\label{e:ch_rescale}
u_t - k^2u_{tyy} + 3kuu_y = k^3(2u_yu_{yy} + uu_{yyy}).
\ee
Additionally rescaling the traveling coordinate $\theta=k(x-ct)=y-\omega t$, where $\omega=kc$ is the temporal frequency, 
it follows; that traveling wave solutions of \eqref{e:ch_rescale} correspond to solutions of the form 
\[
u(y,t)=\phi(y-\omega t),
\]
and hence the given periodic traveling wave $\phi(\theta)$ is necessarily a stationary $1$-periodic solution of the rescaled evolutionary equation
\be
\label{e:stationary_rescale}
u_t - k^2u_{t\theta\theta} - kc(u_{\theta}-k^2u_{\theta\theta\theta}) + 3kuu_{\theta} = k^3(2u_{\theta}u_{\theta\theta} + uu_{\theta\theta\theta}).
\ee 
Note, in particular, that $\phi$ is thereby a $1$-periodic solution of
\be
\label{e:profile_rescale}
-k^2(c-\phi)\phi''+c\phi-\frac{3}{2}\phi^2+\frac{k^2}{2}\left(\phi'\right)^2=E,
\ee
where now $'$ denotes differentiation with respect to $\theta$ and the traveling wave parameters $(a,E,c)$ are now considered
to be functions of $k$, $M$, and $P$.

Now, a straightforward calculation gives that the  linearization of \eqref{e:ch_rescale} about $\phi$ is given by
\[
 v_t=J\mathcal{L}[\phi]v, 
 \]
where
\be
\label{e:Jop}
J:=k\left(1-k^2\partial_\theta^2\right)^{-1}\partial_\theta,
\ee
and
\be
\label{e:Lop}
\mathcal{L}[\phi]:=-k^2\partial_\theta(c-\phi)\partial_\theta + (c-3\phi+k^2\phi'').
\ee 
where these operators are considered on $L^2(\RM)$.  Note, in particular, that the operator $\mathcal{L}[\phi]$ here is the same 
as the Hessian operator $\frac{\partial^2\Lambda}{\partial u^2}$ in \eqref{e:L} rescaled to act on $1$-periodic functions.\footnote{In particular, it is precisely the same
operator arising in our Whitham theory calculations.  See equation \eqref{e:O(e)}.}
As we are interested in the stability of $\phi$ to localized perturbations, this motivates the study of the spectral problem
%This leads us to study the spectral problem
\be
\label{e:specprob}  
\mathcal{A}[\phi]v = \lambda v
\ee 
posed on $L^2(\RM)$, where  here $\lambda\in\CM$ is  a spectral parameter, corresponding to the temporal frequency of the perturbation,
and
\be 
\label{e:Aop} 
\mathcal{A}[\phi] := J\mathcal{L}[\phi]
\ee
is considered as a closed, densely defined linear operator on $L^2(\R)$.   
Motivated by the above considerations, we say that a periodic traveling wave $\phi$ of \eqref{e:ch} is \textit{spectrally unstable} if the $L^2(\R)$-spectrum of $\mathcal{A}[\phi]$ intersects the open right half-plane, i.e., if
$$\sigma_{L^2(\R)}(\mathcal{A}[\phi]) \cap \{\lambda\in\C: {\rm Re}(\lambda)>0\} \ne \emptyset, $$
while it is \textit{spectrally stable} otherwise. This motivates the study of the operator \eqref{e:Aop}.

\begin{remark}\label{r:spec_sym}
We note that since the spectral problem \eqref{e:specprob} is invariant under the transformation $(v,\lambda)\mapsto (\bar{v},\bar{\lambda})$,
as well as the transformation $(\theta,\lambda)\mapsto (-\theta,-\lambda)$, the spectrum of the operator $\mathcal{A}[\phi]$
acting on $L^2(\RM)$ is necessarily symmetric with respect to reflections about the real and imaginary axes.  Consequently, spectral
 stability follows if and only if the $L^2(\RM)$-spectrum of $\mathcal{A}[\phi]$ is purely imaginary.
\end{remark}

As the coefficients of $\mathcal{A}[\phi]$ are $1$-periodic, standard results in Floquet-Bloch theory imply that non-trivial solutions of the 
spectral problem \eqref{e:specprob} cannot be 
square-integrable\footnote{In fact, such solutions can not have finite norm in $L^p(\R)$ for any $1\le p<\infty$.} on $\RM$
and, at best, solutions of \eqref{e:specprob} can be bounded functions on $\RM$: see, for example, \cite{BrHJ16,JP20,RS} and references therein.
It follows that the $L^2(\RM)$ spectrum of $\mathcal{A}[\phi]$ is purely essential with empty point spectrum. 
Further, Floquet-Bloch theory implies that any bounded solution of \eqref{e:specprob} must necessarily be of the form
\begin{equation}\label{e:evec_form}
v(\theta) = e^{i\xi \theta}w(\theta) 
\end{equation}
for some $\xi\in[-\pi,\pi)$ and $w\in L^2_{\rm per}(0,1)$.  
%Further, any bounded solution of \eqref{e:specprob} can be written in the functional form
%%
%\be
%v(x) = e^{i\xi x}w(x) 
%\ee
%%
%for some $w\in L^2_{\rm per}(0,1)$ and $\xi\in[0,2\pi)$. Indeed, it can be shown that $\lambda\in \sigma_{L^2(\R)}(\mathcal{A}[\phi])$ if and only if there exists $\xi\in[0,2\pi)$ such that the eigenvalue problem
%%
%\be
%\label{e:blochprob1}
%\begin{cases}
%\mathcal{A}[\phi]v = \lambda v \\
%v(x+1) = e^{i\xi}v(x),
%\end{cases}
%\ee
%%
%has a non-trivial solution, or, equivalently, if and only if there exists a $\xi\in[0,2\pi)$ and a non-trivial $w\in L^2_{\rm per}(0,1)$ such that
Substituting \eqref{e:evec_form} into \eqref{e:specprob}, it follows that $\lambda\in\CM$ belongs to the $L^2(\RM)$-spectrum of $\mathcal{A}[\phi]$ 
if and only if there exists a $w\in L^2_{\rm per}(0,1)$ and a $\xi\in[-\pi,\pi)$ such that
\be
\label{e:blochprob2}
\lambda w = e^{-i\xi \theta}\mathcal{A}[\phi]e^{i\xi \theta}w =: \mathcal{A}_{\xi}[\phi]w.
\ee 
The one-parameter family of operators $\mathcal{A}_{\xi}[\phi]$ are called the \textit{Bloch operators} associated with $\mathcal{A}[\phi]$, and $\xi$ is referred to as the \textit{Bloch parameter}, or \textit{Bloch frequency}. Since the Bloch operators have compactly embedded domains in $L^2_{\rm per}(0,1)$, for each $\xi\in[-\pi,\pi)$ it follows that that the $L^2_{\rm per}(0,1)$ spectrum of $\mathcal{A}_{\xi}[\phi]$ consists entirely of isolated eigenvalues with finite algebraic multiplicities. In particular, we have the spectral decomposition
\be
\sigma_{L^2(\R)}(\mathcal{A}[\phi]) = \bigcup_{\xi\in[-\pi,\pi)} \sigma_{L^2_{\rm per}(0,1)}(A_{\xi}[\phi]),
\ee
yielding a continuous parameterization of the essential $L^2(\RM)$-spectrum of $\mathcal{A}[\phi]$  by a one-parameter family of $1$-periodic eigenvalue problems for the 
associated Bloch operators $\left\{\mathcal{A}_\xi[\phi]\right\}_{\xi\in[-\pi,\pi)}$.

It follows that to determine the spectral stability of a periodic traveling wave $\phi$ of the CH equation \eqref{e:ch}, one must determine all of the $1$-periodic eigenvalues for each 
Bloch operator $\mathcal{A}_{\xi}[\phi]$ with $\xi\in[-\pi,\pi)$. This is clearly a daunting, if not impossible task in general.  For the purposes of 
modulational stability analysis, however, it is well-known that one need only consider the spectrum of the Bloch operators $\left\{\mathcal{A}_{\xi}[\phi]\right\}_{|\xi|\ll 1}$ in a sufficiently
small neighborhood of the origin in the spectral plane.  In particular, we have the following definition.

\begin{definition}\label{d:MI}
A periodic traveling wave solution $\phi(\cdot;k,M,P)$ of \eqref{e:ch} is said to be modulationally stable if there exists an open neighborhood
$\mathcal{B}\subset\CM$ of the origin $\lambda=0$ and a $\xi_0>0$ such that
\[
\sigma_{L^2_{\rm per}(0,1)}\left(\mathcal{A}_\xi[\phi]\right)\cap\mathcal{B}\subset\RM i
\]
for all $|\xi|<\xi_0$.  The wave $\phi$ is modulationally unstable otherwise.
\end{definition}

In other words, $\phi$ is modulationally stable provided that it is spectrally stable in a sufficiently small neighborhood of the origin.\footnote{In particular, note that while modulational instability
implies spectral instability (in the classical sense), modulational stability does NOT imply spectral stability.} To study the spectrum near the origin,
our strategy is as follows:  We will first study the generalized kernel of the unmodulated Bloch operator $\mathcal{A}_0[\phi]$, showing that it is generically
three-dimensional.  For this, we heavily rely on our assumption that the manifold of periodic traveling wave solutions
can be parameterized by the quantities $(k,M,P)$. 
We then use spectral perturbation theory to track the three eigenvalues $\{\lambda_j(\xi)\}_{j=1}^3$ as they bifurcate 
from the $(\lambda,\xi)=(0,0)$ state.  This is accomplished by constructing asymptotic approximations of the left and right projection operators
onto the total three-dimensional eigenspace associated to the eigenvalues $\lambda_j(\xi)$.  This work will culminate in 
Theorem \ref{t:pert}, encoding the leading order asymptotic expansion of the eigenvalues $\lambda_j(\xi)$ in terms of the eigenvalues
of an explicit matrix $\widehat{\mathcal{D}}_0$.  As such, the eigenvalues of the matrix $\widehat{\mathcal{D}}_0$ defined in \eqref{e:Dmat} determines the 
rigorous modulational instability and stability, as defined in Definition \ref{d:MI} above, of the underlying wave $\phi(\cdot;k,M,P)$.

\begin{remark}\label{r:spec_sym2}
Continuing Remark \ref{r:spec_sym} above, using that $\mathcal{A}[\phi]$ has real-coefficients one can immediately see that
\[
\sigma_{L^2_{\rm per}(0,1)}\left(\mathcal{A}_\xi[\phi]\right)=\overline{\sigma_{L^2_{\rm per}(0,1)}\left(\mathcal{A}_{-\xi}[\phi]\right)}.
\]
Further, since the Bloch operators $\mathcal{A}_\xi[\phi]$ anticommute with the isometry $\mathcal{S}:L^2_{\rm per}(0,1)\to L^2_{\rm per}(0,1)$
given by
\[
\mathcal{S}v(z)=v(-z),
\]
we see that
\[
\sigma_{L^2_{\rm per}(0,1)}\left(\mathcal{A}_{-\xi}[\phi]\right)=-\sigma_{L^2_{\rm per}(0,1)}\left(\mathcal{A}_{\xi}[\phi]\right).
\]
Taken together, it follows that the spectrum of the Bloch operators $\mathcal{A}_\xi[\phi]$ are necessarily symmetric about the imaginary axis.  This observation
will be important later in our rigorous spectral stability analysis.
\end{remark}

\subsection{Analysis of the Unmodulated Operator}\label{s:unmod}

As discussed above, the first step in our  rigorous modulational stability analysis is to understand the $1$-periodic generalized kernel of the 
unmodulated operator $\mathcal{A}_0[\phi]$ defined in \eqref{e:Aop}, as well as the generalized kernel of its adjoint operator\footnote{Information about the kernel of the adjoint is necessary to construct the spectral projections for $\mathcal{A}_0[\phi]$ at $\lambda=0$.} $\mathcal{A}^{\dagger}[\phi]$.  To this end, 
let $\phi=\phi(\cdot;k,M,P)$ be a $1$-periodic solution of the profile equation \eqref{e:profile_rescale}.  Differentiating the profile equation\eqref{e:profile_rescale},
keeping in mind that the traveling wave parameters $(a,E,c)$ are functions of $(k,M,P)$, gives the identities
\begin{equation}\label{e:L_ident}
 \mathcal{L}[\phi]\phi'=0,~~~\mathcal{L}[\phi]\phi_M = E_M -c_M\left(\phi-k^2\phi''\right),~~~
\mathcal{L}[\phi]\phi_P = E_P -c_P\left(\phi-k^2\phi''\right),
\end{equation}
where we recall that $'$ denotes differentiation with respect to $\theta$. Since the frequency $k$ is an independent variable, 
it follows that the functions $\phi'$, $\phi_M$ and $\phi_P$ are all smooth, $1$-periodic functions satisfying
\be
\label{e:AphiM}
%\mathcal{A}[\phi]\phi_M=k(1-k^2\partial_\theta^2)^{-1}\partial_\theta\left(E_M -c_M\left(\phi-k^2\phi''\right)\right)=-kc_M\phi',
  \mathcal{A}[\phi]\phi_M=k(1-k^2\partial_\theta^2)^{-1}\partial_\theta\mathcal{L}[\phi]\phi_M=-kc_M\phi',
\ee
and, similarly,
\be
\label{e:AphiP}
\mathcal{A}[\phi]\phi_P=-kc_P\phi',~~\mathcal{A}[\phi]\phi'=0.
\ee
Additionally, note that
\[
\mathcal{A}^\dag[\phi] = -k\mathcal{L}[\phi]\partial_{\th}(1-k^2\partial_x^2)^{-1}.
\]
Under the assumption that the hypothesis of Proposition \ref{P:kernelL} holds, it follows that
\begin{equation}\label{e:ker_adjoint}
%\mathcal{A}^\dag[\phi]1=0,~~~\mathcal{A}^\dag[\phi](1-k^2\partial_\theta^2)\phi=0.
{\rm ker}\left(\mathcal{A}^\dag[\phi]\right)={\rm span}\left\{1,(1-k^2\partial_\theta^2)\phi\right\}.
\end{equation}
Combining the above observations with the Fredholm alternative yields the following characterization of the generalized right and left
kernels for the linearized operator $\mathcal{A}[\phi]$.

\begin{theorem}
\label{t:co-per}
Let $(a,E,c)\in\mathcal{M}$ and let $\phi(\cdot;a,E,c)$ be the corresponding $T=T(a,E,c)$-periodic traveling wave solution of \eqref{e:ch}.  Assume
that the non-degeneracy condition \eqref{e:invert} holds at $\phi$ and, further, assume $T_a\neq 0$ and that
\begin{equation}\label{e:nondeg}
\left(M_c\right)^2+\left(P_c\right)^2\neq 0.
\end{equation}
%Let $\phi=\phi(\cdot;k,M,P)$ be a $2\pi$-periodic traveling wave solution\footnote{Implicit here is the assumption that the non-degeneracy
%condition \ref{e:invert} holds at $\phi$.} of the profile equation \eqref{e:profile_rescale}.  Further, assume that the 
Then $\lambda=0$ is a $1$-periodic eigenvalue of $\mathcal{A}[\phi]$ with algebraic multiplicity three and geometric multiplicity
two.  In particular, reparameterizing the solution as $\phi(\cdot;k,M,P)$ and defining
\begin{align*}
&\Phi_1^0:=\phi'\qquad \Phi_2^0:=\phi_M\qquad \Phi_3^0:=\phi_P\\
&\Psi_1^0:=\beta\qquad \Psi_2^0:=1\qquad \Psi_3^0:=\phi-k^2\phi''
\end{align*}
where $\beta$ is the unique $1$-periodic odd function satisfying $\mathcal{A}^\dag[\phi]\beta\in{\rm ker}\left(\mathcal{A}^\dag[\phi]\right)$
and $\left<\beta,\Phi_1^0\right>=1$, we have that $\{\Phi_\ell^0\}_{\ell=1}^3$ and $\{\Psi_j^0\}_{j=1}^3$ provide a basis of solutions for
the generalized $2\pi$-periodic kernels of $\mathcal{A}[\phi]$ and $\mathcal{A}^\dag[\phi]$, respectively.  In particular, we 
have $\left<\Psi_j^0,\Phi_\ell^0\right>=\delta_{j\ell}$ and the $\Phi_\ell^0$ and $\Psi_j^0$ satisfy the equations
\[
\mathcal{A}[\phi]\Phi_1^0=0,~~~\mathcal{A}[\phi]\Phi_2^0=-kc_M\Phi_1^0,~~~ \mathcal{A}[\phi]\Phi_3^0=-kc_P\Phi_1^0
\]
and
\[
\mathcal{A}^\dag[\phi]\Psi_2^0=0=\mathcal{A}^\dag[\phi]\Psi_3^0,~~~ \mathcal{A}^\dag[\phi]\Psi_1^0\in{\rm span}\left\{\Psi_2^0,\Psi_3^0\right\}\setminus\{0\}.
\]
\end{theorem}
%
%\begin{remark}
%{\color{blue} Question -- If $c_M=c_P=0$ then $E_P\phi_M-E_M\phi_P$ is an even function in kernel of $\mathcal{L}[\phi]$, which would imply (if $T_a\neq 0$)
%that $\phi_M$ is a constant multiple of $\phi_E$...  is that bad?}
%\end{remark}

\begin{proof}
By hypothesis, we know that the kernel of $\mathcal{L}[\phi]$ defined in \eqref{e:Lop} is precisely $\phi'$ and, furthermore, that the kernel 
of the adjoint operator $\mathcal{A}^\dag[\phi]$ is given as in \eqref{e:ker_adjoint}.  Using the identities in \eqref{e:L_ident} it follows that
\[
\ker\left(\mathcal{A}[\phi]\right)={\rm span}\left\{\phi',c_M\phi_P-c_P\phi_M\right\}.
\]
To study the generalized kernel of $\mathcal{A}[\phi]$, we note by the Fredholm alternative that the equation
\[
\mathcal{A}[\phi]\psi=c_M\phi_P-c_P\phi_M
\]
has no $1$-periodic solution $\psi$ provided that either $c_M\neq 0$ or $c_P\neq 0$, i.e. provided that \eqref{e:nondeg} holds.  Indeed, note that
\[
\left<1,c_M\phi_P-c_P\phi_M\right>=c_MM_P-c_PM_M=-c_P
\]
and, similarly
\[
\left<(1-k^2\partial_\theta^2)\phi ,c_M\phi_P-c_P\phi_M\right>=c_MP_P-c_PM_P=c_M.
\]
Here, we are using that the $M$ and $P$ are independent variables in our parameterization of $\phi(\cdot;k,M,P)$.

Further, from \eqref{e:L_ident} we know that 
\[
\mathcal{A}[\phi]\Phi_2^0=-kc_M\phi'.
\]
Noting that
\[
\left<1,\Phi_2\right>=1
\]
by considerations as above, it follows that
\[
\ker\left(\mathcal{A}[\phi]^3\right)\setminus\ker\left(\mathcal{A}[\phi]^2\right)=\emptyset.
\]
Taken together, this establishes that $\lambda=0$ is indeed a $1$-periodic eigenvalue of $\mathcal{A}[\phi]$ of algebraic multiplicity three and geometric multiplicity
two.  The existence of a function $\beta$ satisfying 
\[
\mathcal{A}^\dag[\phi]\beta\in{\rm span}\left\{\Psi_2^0,\Psi_3^0\right\}\setminus\{0\}
\]
similarly follows using the Fredholm Alternative.
\end{proof}

Throughout the remainder of our work, we will assume the hypotheses of Theorem \ref{t:co-per} hold.
%
%Before moving on, we note that
%\[
%\mathcal{A}_0[\phi]\phi_k=-kc_k\phi'+...
%\]

\subsection{Spectral Perturbation Theory}

Now that we have constructed a basis for the generalized kernels of $\mathcal{A}[\phi]$, and its adjoint $\mathcal{A}^{\dagger}[\phi]$ in coordinates compatible with the 
Whitham modulation system \eqref{e:Wsys}, we now consider the spectrum of the Bloch operators for $|(\lambda,\xi)|\ll 1$.  To this end, we observe
that the Bloch operators in \eqref{e:blochprob2} can be decomposed as
\be
\mathcal{A}_{\xi}[\phi] = J_{\xi}\mathcal{L}_{\xi}[\phi],
\ee
where here $J_{\xi}:=e^{-i\xi\theta}Je^{i\xi\theta}$ and  $\mathcal{L}_{\xi}[\phi]:=e^{-i\xi\theta}\mathcal{L}[\phi]e^{i\xi\theta}$  are the Bloch operators
associated with the operators $J$ and $\mathcal{L}[\phi]$ in \eqref{e:Jop} and \eqref{e:Lop}, respectively.  Recall all the Bloch operators
are considered to be densely defined on $L^2_{\rm per}(0,1)$.  Our goal here is to study the  $1$-periodic eigenvalue problem
\[
J_\xi\mathcal{L}_\xi[\phi]v=\lambda v
\]
for $|(\lambda,\xi)|\ll 1$.  

To this end, we begin by expanding the Bloch operators for $|\xi|\ll 1$.  Since the Bloch operators are analytic in $\xi$ it is easy to verify that
\begin{align*}
\mathcal{L}_\xi[\phi]&= -k^2(\partial_\theta+i\xi)\left(c-\phi\right)(\partial_\theta+i\xi)+\left(c-3\phi+k^2\phi''\right)\\
&= L_0+(ik\xi)L_1+(ik\xi)^2L_2,
\end{align*}
where
$$ L_0=\mathcal{L}[\phi],~~~L_1=-k\partial_\theta\left((c-\phi)\cdot\right)-k(c-\phi)\partial_\theta,~~~L_2=-(c-\phi),
$$
are operators acting on $L^2_{\rm per}(0,1)$.  Regarding the Bloch expansion for $J_{\xi}$, first note that
$$ 1-k^2(\partial_\theta+i\xi)^2 = \mathcal{G}_0 + (ik\xi)\mathcal{G}_1 + (ik\xi)^2\mathcal{G}_2, $$
where
$$ \mathcal{G}_0=1-k^2\partial_\theta^2,~~~\mathcal{G}_1=-2k\partial_\theta,~~~\mathcal{G}_2=-1, $$
are operators acting on $L^2_{\rm per}(0,1)$.  Since $\mathcal{G}_0$ is invertible in $L^2_{\rm per}(0,1)$, we can rewrite the above expansion as
$$ 1-k^2(\partial_\theta+i\xi)^2 = \left(I+(ik\xi)\mathcal{G}_1\mathcal{G}_0^{-1}+(ik\xi)^2\mathcal{G}_2\mathcal{G}_0^{-1}\right)\mathcal{G}_0, $$
and hence the (nonlocal) operator $(1-k^2(\d_\theta+i\xi)^2)^{-1}$ can be expanded via a Neuman series as 
\begin{align*}
\left(1-k^2(\partial_\theta+i\xi)^2\right)^{-1}&=\mathcal{G}_0^{-1}\sum_{j=0}^\infty (-1)^j\left[\left((ik\xi)\mathcal{G}_1 + (ik\xi)^2\mathcal{G}_2\right)\mathcal{G}_0^{-1}\right]^j\\
&=\mathcal{G}_0^{-1}-\mathcal{G}_0^{-1}\left((ik\xi)\mathcal{G}_1\mathcal{G}_0^{-1}+(ik\xi)^2\mathcal{G}_2\mathcal{G}_0^{-1}\right)\\
&\quad+\mathcal{G}_0^{-1}\left((ik\xi)^2\mathcal{G}_1\mathcal{G}_0^{-1}\mathcal{G}_1\mathcal{G}_0^{-1}\right)+\mathcal{O}(\xi^3)\\
&=G_0+(ik\xi)G_1+(ik\xi)^2G_2+\mathcal{O}(\xi^3),
\end{align*}
where
\begin{align*}
G_0=(1-k^2\partial_\theta^2)^{-1},~~~G_1=-\mathcal{G}_0^{-1}\mathcal{G}_1\mathcal{G}_0^{-1}=2k(1-k^2\partial_\theta^2)^{-2}\partial_\theta,
\end{align*}
and
\begin{align*}
G_2&=-\mathcal{G}_0^{-1}\mathcal{G}_2\mathcal{G}_0^{-1}+\mathcal{G}_0^{-1}\mathcal{G}_1\mathcal{G}_0^{-1}\mathcal{G}_1\mathcal{G}_0^{-1}\\
&=(1-k^2\partial_\theta^2)^{-2}+4k^2(1-k^2\partial_\theta^2)^{-3}\partial_\theta^2.
\end{align*}
Taken together, for $|\xi|\ll 1$ we have the expansion
\begin{align*}
J_\xi &= k\left(1-k^2(\partial_\theta+i\xi)^2\right)^{-1}\left(\partial_\theta+i\xi\right)\\
&=k\left(G_0+(ik\xi)G_1+(ik\xi)^2G_2+\mathcal{O}(\xi^3)\right)\left(\partial_\theta+i\xi\right)\\
&=J_0+(ik\xi)J_1+(ik\xi)^2J_2+\mathcal{O}(\xi)^3,
\end{align*}
where
\begin{equation}\label{e:J_expand}
\left\{\begin{aligned}
J_0&=kG_0\partial_\theta = k\left(1-k^2\partial_\theta^2\right)^{-1}\partial_\theta,\\
J_1&=kG_1\partial_\theta+G_0=2k^2(1-k^2\partial_\theta^2)^{-2}\partial_\theta^2+(1-k^2\partial_\theta^2)^{-1},\\
J_2&=kG_2\partial_\theta + G_1=k(1-k^2\partial_\theta^2)^{-2}\partial_\theta + 4k^3(1-k^2\partial_\theta^2)^{-3}\partial_\theta^3+2k(1-k^2\partial_\theta^2)^{-2}\partial_\theta.
\end{aligned}\right.
\end{equation}
Finally, it follows that for $|\xi|\ll 1$ we have the analytic expansions
\begin{align*}
\mathcal{A}_\xi[\phi]&= J_\xi\mathcal{L}_\xi[\phi]\\
&=\left(J_0+(ik\xi)J_1+(ik\xi)^2J_2\right)\left(L_0+(ik\xi)L_1+(ik\xi)^2L_2\right)+\mathcal{O}(\xi^3)\\
&=A_0+(ik\xi)A_1+(ik\xi)^2A_2+\mathcal{O}(\xi^3)
\end{align*}
where
\be
\label{e:Aidentities}
A_0=J_0L_0,~~A_1=J_0L_1+J_1L_0,~~A_2=J_0L_2+J_1L_1+J_2L_0.
\ee 

With the above analytic expansions in hand, we recall by Theorem \ref{t:co-per} that $\lambda=0$ is a $1$-periodic 
eigenvalue of $\mathcal{A}_0[\phi]$ with algebraic multiplicity three and geometric multiplicity two.  
Noting that $\mathcal{A}_\xi[\phi]$ is a relatively compact perturbation of $\mathcal{A}_0[\phi]$ for all $\xi\in[-\pi,\pi)$,
it follows by spectral perturbation theory that the operator $\mathcal{A}_\xi[\phi]$ will have 
three eigenvalues $\{\lambda_j(\xi)\}_{j=1}^3$ near $\lambda=0$ for $|\xi|\ll 1$.  These eigenvalues necessarily
bifurcate from the $\lambda_j(0)=0$ state, and the modulational stability or instability of the underlying wave $\phi$
is thus determined by tracking these three eigenvalues for $|\xi|\ll 1$ and determining if they remain confined to the 
imaginary axis or not.   This will be accomplished by projecting the infinite dimensional eigenvalue problem
\[
\mathcal{A}_\xi[\phi]v=\lambda v,
\]
for $|\xi|\ll 1$ onto the total three-dimensional eigenspace associated with the eigenvalues $\{\lambda_j(\xi)\}_{j=1}^3$ bifurcating
from the $(\lambda,\xi)$ state.

To this end, we note that the standard spectral perturbation theory  (see, for example, Theorems 1.7 and 1.8 in \cite[Chapter VII.1.3]{K76}) guarantees that the dual bases 
in Theorem \ref{t:co-per} extend analytically for $|\xi|\ll 1$ into dual right and left
bases $\left\{\Phi^\xi_\ell\right\}_{\ell=1}^3$ and $\left\{\Psi^\xi_j\right\}_{j=1}^3$ associated to the three eigenvalues $\{\lambda_j(\xi)\}_{j=1}^3$
near the origin and, further, these dual bases satisfy $\langle\Psi_j^\xi,\Phi_\ell^\xi\rangle=\delta_{j\ell}$ for all $|\xi|\ll 1$. 
Given the existence of these bases, for $|\xi|\ll 1$ we can construct $\xi$-dependent rank three eigenprojections, namely, 
\begin{equation}\label{e:proj}
\Pi(\xi):L^2_{\rm per}(0,1)\to\bigoplus_{j=1}^3\ker\Big(\mathcal{A}_\xi[\phi]-\lambda_j(\xi)I\Big),~~
\widetilde{\Pi}(\xi):L^2_{\rm per}(0,1)\to\bigoplus_{j=1}^3\ker\Big(\mathcal{A}^\dag_\xi[\phi]-\overline{\lambda_j(\xi)}I\Big),
\end{equation}
whose ranges are precisely the total right and left eigenspaces associated with the eigenvalues $\{\lambda_j(\xi)\}_{j=1}^3$
near the origin.  In particular, we note that the small eigenvalues $\lambda_j(\xi)$ are given precisely by the 
roots of the characteristic polynomial
\[
\det\left(\mathcal{D}_\xi-\lambda I\right)=0
\]
associated with the $3\times 3$ matrix operator
\be
\label{e:matrix_def}
\mathcal{D}_\xi:=\widetilde{\Pi}(\xi)\mathcal{A}_\xi[\phi]\Pi(\xi)=\left[\left<\Psi_j^\xi,\mathcal{A}_\xi[\phi]\Phi_\ell^\xi\right>\right]_{j,\ell=1}^3,
\ee
defined for $|\xi|\ll 1$, where here $I$ denotes the $3\times 3$ identity matrix.\footnote{Here, we are using that $\widetilde{\Pi}(\xi)\Pi(\xi)=I$, by construction.}
It remains to calculate an expansion of the matrix $\mathcal{D}_\xi$ to sufficiently high order.

To this end, we begin by expanding the right and left dual bases as
\begin{align*}
\Phi_\ell^\xi&=\Phi_\ell^0 + (ik\xi)\left(\frac{1}{ik}\partial_\xi\Phi_\ell^\xi\Big{|}_{\xi=0}\right)+(ik\xi)^2\left(\frac{1}{(ik)^2}\partial_\xi^2\Phi_\ell^\xi\Big{|}_{\xi=0}\right)+\mathcal{O}(\xi^3),\\
\Psi_j^\xi&=\Psi_j^0 + (ik\xi)\left(\frac{1}{ik}\partial_\xi\Psi_j^\xi\Big{|}_{\xi=0}\right)+(ik\xi)^2\left(\frac{1}{(ik)^2}\partial_\xi^2\Psi_j^\xi\Big{|}_{\xi=0}\right)+\mathcal{O}(\xi^3),
\end{align*}
thus yielding the asymptotic expansion
\[
\mathcal{D}_\xi = D_0 + (ik\xi)D_1 + (ik\xi)^2 D_2 + \mathcal{O}(\xi^3)
\]
for $|\xi|\ll 1$.  In what follows, we seek to determine the leading order asymptotics for each entry of the $3\times 3$ matrix $\mathcal{D}_\xi$ in the 
limit as $\xi\to 0$.

To begin, note that Theorem \ref{t:co-per} immediately implies that
\be
\label{e:D0}
D_0=\left(\begin{array}{c|cc}
								0 & -kc_M & -kc_P\\
								\hline 0 & 0 & 0\\
								0 & 0 & 0
								\end{array}\right),
\ee
where here the horizontal and vertical lines are included simply as a visual aide to organize the $3\times 3$ matrix into appropriate sub-blocks.
Further, an elementary calculation shows that
\begin{equation}\label{e:D1_pre}
D_1=\left[\left<\Psi_j^0,A_0\frac{1}{ik}\partial_\xi\Phi_\ell^\xi\Big{|}_{\xi=0}+A_1\Phi_\ell^0\right>+\left<\frac{1}{ik}\partial_\xi\Psi_j^\xi\Big{|}_{\xi=0},A_0\Phi_\ell^0\right>\right]_{j,\ell=1}^3.
\end{equation}
In order to more explicitly calculate the relevant entries of $D_1$, we make two important observations.  First, we note that the normalization of the right and left bases imply
that
\[
0=\partial_{\xi}\left<\Psi_j^\xi,\Phi_\ell^\xi\right>\Big{|}_{\xi=0}
\]
and hence
\begin{equation}\label{mv_der}
\left<\frac{1}{ik}\partial_\xi\Psi_j^\xi\Big{|}_{\xi=0},\Phi_\ell^0\right>=-\left<\Psi_j^0,\frac{1}{ik}\partial_\xi\Phi_\ell^\xi\Big{|}_{\xi=0}\right>.
\end{equation}
Secondly, we note that since $\Phi_1^\xi$ is in the total right eigenspace of $\mathcal{A}_\xi[\phi]$ 
associated with the eigenvalues $\{\lambda_j(\xi)\}_{j=1}^3$ near the origin, it follows that 
$\mathcal{A}_\xi[\phi]\Phi_1^\xi$ is invariant with respect to the projection $\Pi(\xi)$ defined in \eqref{e:proj}.
Specifically, we have
\[
\Pi(\xi)\mathcal{A}_\xi[\phi]\Phi_1^\xi = \mathcal{A}_\xi[\phi]\Phi_1^\xi
\]
for all $|\xi|\ll 1$, and hence differentiating with respect to $\xi$ and evaluating at $\xi=0$ yields
\[
\frac{1}{ik}\partial_\xi\Pi(\xi)\Big{|}_{\xi=0}A_0\Phi_1^0 + \Pi(0)\left(A_1\Phi_1^0 + A_0\left(\frac{1}{ik}\partial_\xi\Phi_1^\xi\Big{|}_{\xi=0}\right)\right)
%=A_1\Phi_1^0 + A_1\Phi_1^0.
=A_1\Phi_1^0 + A_0\left(\frac{1}{ik}\partial_\xi\Phi_1^\xi\Big{|}_{\xi=0}\right).
\]
Noting that $A_0\Phi_1^0=0$ by Theorem \ref{t:co-per} and also that differentiating the profile equation \eqref{e:profile_rescale} with respect to $k$ gives 
\begin{equation}\label{e:Aphik}
A_0\phi_k = \mathcal{A}[\phi]\phi_k = -kc_k\phi'-A_1\phi',
\end{equation}
it follows that the above can be rewritten as
\[
\Pi(0)\left(A_0\left(\frac{1}{ik}\partial_\xi\Phi_1^\xi\Big{|}_{\xi=0}-\phi_k\right)\right)=A_0\left(\frac{1}{ik}\partial_\xi\Phi_1^\xi\Big{|}_{\xi=0}-\phi_k\right).
\]
and hence that the function
\[
\frac{1}{ik}\partial_\xi\Phi_1^\xi\Big{|}_{\xi=0}-\phi_k
\]
necessarily lies in the generalized eigenspace for $A_0$.  By Theorem \ref{t:co-per} it follows that 
\[
\frac{1}{ik}\partial_\xi\Phi_1^\xi\Big{|}_{\xi=0}=\phi_k+\sum_{j=1}^3 a_j\Phi_j^0
\]
for some constants $a_j\in\CM$.  In particular, by replacing the function $\Phi_1^\xi$ with the function
\begin{equation}\label{e:mod_basis1}
\widetilde{\Phi}_1^\xi = \Phi_1^\xi - (ik\xi)\sum_{j=1}^3 a_j\Phi_j^0,
\end{equation}
while simultaneously replacing the left bases $\{\Psi_j^\xi\}_{j=1}^3$ with the functions
\[
\widetilde{\Psi}_j^\xi = \Psi_j^\xi + (ik\xi)a_j\Psi_1^\xi,~~j=1,2,3,
\]
it immediately follows that
\[
\widetilde{\Phi}_1^\xi = \Phi_1^0 + (ik\xi)\phi_k  + \mathcal{O}(\xi^2),~~~I_\xi:=\widetilde{\Pi}(\xi)\Pi(\xi) = I+\mathcal{O}(\xi^2).
\]
Consequently, we see that up to a harmless modification of the right and left basis functions, we can identify that first variation of $\Phi_1^\xi$ in $\xi$ at $\xi=0$ 
as precisely $\phi_k$, while still retaining the biorthonormality of the basis up to $\mathcal{O}(\xi^2)$.  
As we will see, this order of preserving biorthonomality is sufficient for our calculations. 
Further, the  identity \eqref{mv_der} still holds as well as it only involves first order information in $\xi$.

%
%{\color{blue}
%Before continuing, we make some important observations.  First, differentiating the profile equation with respect to $k$ gives
%%
%\be
%\label{e:Aphik}
%A_0\phi_k = \mathcal{A}[\phi]\phi_k = -kc_k\phi'-A_1\phi'.
%\ee
%%
%Further, noting that $M$ and $P$ are independent variables with respect to $k$ we see that
%%
%\[
%\left<1,\phi_k \right>=M_k=0,
%\]
%%
%and
%%
%\begin{align*}
%\left<\phi-k^2\phi'',\phi_k\right>&=\int_0^1\phi\phi_k+k^2\phi'\phi'_k~dx\\
%&=\frac{1}{2}\partial_k\left(\int_0^1\phi^2+k^2(\phi')^2dx\right)-k\left<\phi',\phi'\right>\\
%&=P_k-k\left<\phi',\phi'\right> \\
%&=-k\left<\phi',\phi'\right>.
%\end{align*}
%%
%}

With the above preparations, we can now calculate the relevant entries of the matrix $D_1$.  First, note  from \eqref{e:D1_pre} that the  $(1,1)$ entry is given by
\begin{align*}
\left<\Psi_1^0,A_0\frac{1}{ik}\partial_\xi\Phi_1^\xi\Big{|}_{\xi=0}+A_1\Phi_1^0\right>+\left<\frac{1}{ik}\partial_\xi\Psi_1^\xi\Big{|}_{\xi=0},A_0\Phi_1^0\right>.
\end{align*}
Using \eqref{e:mod_basis1} as well as the identity \eqref{e:Aphik} we have
\begin{align*}
\left<\Psi_1^0,A_0\frac{1}{ik}\partial_\xi\Phi_1^\xi\Big{|}_{\xi=0}\right>&=\left<\Psi_1^0,-kc_k\phi'-A_1\phi'\right>\\
&=-kc_k\left<\Psi_1^0,\Phi_1^0\right>-\left<\Psi_1^0,A_1\Phi_1^0\right>
\end{align*}
and hence, using that $A_0\Phi_1^0=0$, it follows that
\begin{align*}
\left<\Psi_1^0,A_0\frac{1}{ik}\partial_\xi\Phi_1^\xi\Big{|}_{\xi=0}+A_1\Phi_1^0\right>+\left<\frac{1}{ik}\partial_\xi\Psi_1^\xi\Big{|}_{\xi=0},A_0\Phi_1^0\right>&=-kc_k.
\end{align*}
Similarly, for $j=2,3$ we have
\begin{align*}
&\left<\Psi_j^0,A_0\frac{1}{ik}\partial_\xi\Phi_1^\xi\Big{|}_{\xi=0}+A_1\Phi_1^0\right>+\left<\frac{1}{ik}\partial_\xi\Psi_j^\xi\Big{|}_{\xi=0},A_0\Phi_1^0\right>\\
&\qquad=\left<A_0^\dag\Psi_j^0,\frac{1}{ik}\partial_\xi\Phi_1^\xi\Big{|}_{\xi=0}\right>+\left<\Psi_1^0,A_1\Phi_1^0\right>=0,
\end{align*}
where the first term vanishes since $A_0^\dag\Psi_j^0=0$ for $j=2,3$ and the second term vanishes by parity.\footnote{Note that $\Phi_1^0$ is odd, $A_1$ preserves
parity, and $\Psi_j^0$ is even for $j=2,3$.}

Continuing, note that for $\ell=2,3$ we have
\begin{align*}
&\left<\Psi_2^0,A_0\frac{1}{ik}\partial_\xi\Phi_\ell^\xi\Big{|}_{\xi=0}+A_1\Phi_\ell^0\right>+\left<\frac{1}{ik}\partial_\xi\Psi_2^\xi\Big{|}_{\xi=0},A_0\Phi_\ell^0\right>=
\left<\Psi_2^0,A_1\Phi_\ell^0\right>+\left<\frac{1}{ik}\partial_\xi\Psi_2^\xi\Big{|}_{\xi=0},A_0\Phi_\ell^0\right>
%&\qquad=\left<\Psi_2^0,A_1\Phi_\ell^0\right>+kc_?\left<\frac{1}{ik}\partial_\xi\Psi_2^\xi\Big{|}_{\xi=0},\phi'\right>,
\end{align*}
where we have used that $A_0^\dag\Psi_2^0=0$.
For the second term above, note that if $\ell=2$ we have
\begin{align*}
\left<\frac{1}{ik}\partial_\xi\Psi_2^\xi\Big{|}_{\xi=0},A_0\Phi_2^0\right>&=-kc_M\left<\frac{1}{ik}\partial_\xi\Psi_2^\xi\Big{|}_{\xi=0},\Phi_1^0\right>\\
&=kc_M\left<\Psi_2^0,\frac{1}{ik}\partial_\xi\Phi_1^\xi\Big{|}_{\xi=0}\right>\\
&=kc_M\left<\Psi_2^0,\phi_k\right>\\
&=kc_M\left<1,\phi_k\right>,
%&=kc_M M_k = 0.
\end{align*}
where the second equality follows by \eqref{mv_der}.   Noting that $M$ is an independent variable with respect to $k$, we have
\[
\left<1,\phi_k\right>=M_k=0
\]
and hence 
\[
\left<\frac{1}{ik}\partial_\xi\Psi_2^\xi\Big{|}_{\xi=0},A_0\Phi_2^0\right>=0.
\]
Similarly, for $\ell=3$ we have
\[
\left<\frac{1}{ik}\partial_\xi\Psi_2^\xi\Big{|}_{\xi=0},A_0\Phi_3^0\right>=kc_P\left<1,\phi_k\right>=kc_P M_k=0
\]
and hence for $\ell=2,3$ we have
\begin{align*}
&\left<\Psi_2^0,A_0\frac{1}{ik}\partial_\xi\Phi_\ell^\xi\Big{|}_{\xi=0}+A_1\Phi_\ell^0\right>+\left<\frac{1}{ik}\partial_\xi\Psi_2^\xi\Big{|}_{\xi=0},A_0\Phi_\ell^0\right>=
\left<\Psi_2^0,A_1\Phi_\ell^0\right>.
\end{align*}

To complete our computation of $D_1$, we follow the above strategy and note that for $\ell=2,3$ we have
\begin{align*}
\left<\Psi_3^0,A_0\frac{1}{ik}\partial_\xi\Phi_\ell^\xi\Big{|}_{\xi=0}+A_1\Phi_\ell^0\right>+\left<\frac{1}{ik}\partial_\xi\Psi_3^\xi\Big{|}_{\xi=0},A_0\Phi_\ell^0\right>=
\left<\Psi_3^0,A_1\Phi_\ell^0\right>+\left<\frac{1}{ik}\partial_\xi\Psi_3^\xi\Big{|}_{\xi=0},A_0\Phi_\ell^0\right>,
\end{align*}
where we have used $A_0^\dag\Psi_3^0=0$.  
As above, for the second term above we note that when $\ell=2$ we have
\begin{align*}
\left<\frac{1}{ik}\partial_\xi\Psi_3^\xi\Big{|}_{\xi=0},A_0\Phi_2^0\right>&=-kc_M\left<\frac{1}{ik}\partial_\xi\Psi_3^\xi\Big{|}_{\xi=0},\Phi_1^0\right>\\
&=kc_M\left<\Psi_3^0,\phi_k\right>\\
&=kc_M\left<\phi-k^2\phi'',\phi_k\right>.
%&=-kc_M\left<\phi',\phi'\right>.
\end{align*}
Noting that $P$ is an independent variable with respect to $k$, a direct calculation using the product  rule shows that
\begin{align*}
\left<\phi-k^2\phi'',\phi_k\right>&=\int_0^1\phi\phi_k+k^2\phi'\phi'_k~dx\\
%&=\frac{1}{2}\partial_k\left(\int_0^1\phi^2+k^2(\phi')^2dx\right)-k\left<\phi',\phi'\right>\\
&=P_k-k\left<\phi',\phi'\right> \\
&=-k\left<\phi',\phi'\right>
\end{align*}
and hence
\[
\left<\frac{1}{ik}\partial_\xi\Psi_3^\xi\Big{|}_{\xi=0},A_0\Phi_2^0\right>=-k^2c_M\left<\phi',\phi'\right>.
\]
Similarly, for $\ell=3$ we have
\[
\left<\frac{1}{ik}\partial_\xi\Psi_3^\xi\Big{|}_{\xi=0},A_0\Phi_2^0\right>=-k^2c_P\left<\phi',\phi'\right>.
\]
Putting everything together, it follows that
\be
\label{e:D1}
D_1=\left(\begin{array}{c|cc}
			-kc_k & * & *\\
\hline 	0 & \left<\Psi_2^0,A_1\Phi_2^0\right> & 	\left<\Psi_2^0,A_1\Phi_3^0\right>\\
			0 & \left<\Psi_3^0,A_1\Phi_2^0\right>-k^2c_M\left<\phi',\phi'\right> &  \left<\Psi_3^0,A_1\Phi_3^0\right>-k^2c_P\left<\phi',\phi'\right> 
			\end{array}\right)
\ee 
where here $*$ represents un-computed entries that do not contribute to the leading order asymptotics\footnote{Recall we only seek here the leading
order (non-zero) asymptotics of the entries of the matrix $\mathcal{D}_\xi$ in the limit as $\xi\to 0$, and the entries with the $*$ were already seen in \eqref{e:D0}
to have $\mathcal{O}(1)$ contributions.} of $\mathcal{D}_\xi$.  

Taking \eqref{e:D0} and \eqref{e:D1} together, it remains to calculate two entries of the matrix $D_2$.  By a direct calculation,
we find that
\begin{align*}
D_2&=\left[\left<\Psi_j^0,A_0\frac{1}{(ik)^2}\partial_\xi^2\Phi_\ell^\xi\Big{|}_{\xi=0} + A_1\frac{1}{ik}\partial_\xi\Phi_\ell^\xi\Big{|}_{\xi=0}+A_2\Phi_\ell^0\right>\right]_{j,\ell=1}^3\\
&\quad+\left[\left<\frac{1}{ik}\partial_\xi\Psi_j^\xi\Big{|}_{\xi=0},A_0\frac{1}{ik}\partial_\xi\Phi_\ell^\xi\Big{|}_{\xi=0}+A_1\Phi_\ell^0\right>
+\left<\frac{1}{(ik)^2}\partial_\xi^2\Psi_j^\xi\Big{|}_{\xi=0},A_0\Phi_\ell^0\right>\right]_{j,\ell=1}^3.
\end{align*}
Based on the structure of $D_0$ and $D_1$, we need to compute the above when $j=1$ and  $\ell=2,3$.  For these cases, following the same strategies
as above we find that
the first term is zero (moving the adjoint of $A_0$ over), the second term simplifies to  $\left<\Psi_j^0,A_1\phi_k\right>$, and the third
term just doesn't simplify.  Similarly, the fourth term becomes
\[
\left<\frac{1}{ik}\partial_\xi\Psi_j^\xi\Big{|}_{\xi=0},A_0\phi_k\right>,
\]
while the fifth term is
\[
\left<\frac{1}{ik}\partial_\xi\Psi_j^\xi\Big{|}_{\xi=0},-kc_k\Phi_1^0-A_0\phi_k\right>=kc_k\left<\Psi_j^0,\phi_k\right>-\left<\frac{1}{ik}\partial_\xi\Psi_j^\xi\Big{|}_{\xi=0},A_0\phi_k\right>.
\]
Finally, the sixth term is clearly zero. It follows that the relevant entries of $D_2$ are given by
\be
\label{e:D2} 
D_2=\left(\begin{array}{c|cc}
				*  & * & *\\
\hline			\left<\Psi_2^0,A_1\phi_k+A_2\Phi_1^0\right> & * & *\\
				\left<\Psi_3^0,A_1\phi_k+A_2\Phi_1^0\right> + kc_k\left<\Psi_3^0,\phi_k\right> & * & *
				\end{array}\right)
\ee			

Putting \eqref{e:D0}, \eqref{e:D1} and \eqref{e:D2} together, it follows that the matrix $\mathcal{D}_\xi$ in \eqref{e:matrix_def} expands analytically for $|\xi|\ll 1$ 
with asymptotic orders given by
\[
\mathcal{D}_\xi=\left(\begin{array}{c|ccc}
				\mathcal{O}(\xi) &  & \mathcal{O}(1) & \\
\hline			  & & &\\
				\mathcal{O}(\xi^2) &  & \mathcal{O}(\xi) & \\
				  & & & 
				  \end{array}\right),
\]	
where the upper-left block is $1\times 1$. By standard spectral perturbation theory, it follows that the eigenvalues $\lambda_j(\xi)$ of $\mathcal{D}_\xi$ are
at least $C^1$ in $\xi$ for $|\xi|\ll 1$, and satisfy $\lambda_j(0)=0$. As such, the eigenvalues $\lambda_j$ can be factored as
\be
\label{e:ev_factor}
\lambda_j(\xi)=ik\xi\mu_j(\xi)
\ee
for some continuous functions $\mu_j$ defined for $|\xi|\ll 1$.  Further, we note that by defining the invertible matrix
\be
\label{e:Smat}
S(\xi):=\left(\begin{array}{c|cc}
				ik\xi & 0 & 0\\
\hline			0 & 1 & 0\\
				0 & 0 & 1
				\end{array}\right) 				
\ee
and setting
%%
%\[
%{\color{red}\widehat{\mathcal{D}}_\xi:=\frac{ik\xi}S(\xi)\mathcal{D}_\xi S(\xi)^{-1},~~\widehat{I}_\xi:=S(\xi)I_\xi S(\xi)^{-1}, \quad {\rm \color{red}divide\, by\, S?}}
%\]
\[
\widehat{\mathcal{D}}_\xi:=\frac{1}{ik\xi}S(\xi)\mathcal{D}_\xi S(\xi)^{-1},~~\widehat{I}_\xi:=S(\xi)I_\xi S(\xi)^{-1}
\]
it follows that both $\widehat{\mathcal{D}}_\xi$ and $\widehat{I}_\xi$ are analytic in $ik\xi$ for $|\xi|\ll 1$ and satisfy $\widehat{I}_0=I$ and
\be
\label{e:Dmat} 
\widehat{\mathcal{D}}_0=
				\left(\begin{array}{c|cc}
				-kc_k & -kc_M & -kc_P\\
\hline			\left<\Psi_2^0,A_1\phi_k+A_2\Phi_1^0\right> &   \left<\Psi_2^0,A_1\Phi_2^0\right> & 	\left<\Psi_2^0,A_1\Phi_3^0\right>\\
				\left<\Psi_3^0,A_1\phi_k+A_2\Phi_1^0+kc_k\phi_k\right> &   \left<\Psi_3^0,A_1\Phi_2^0\right>-k^2c_M\left<\phi_{\th},\phi_{\th}\right> &  \left<\Psi_3^0,A_1\Phi_3^0\right>-k^2c_P\left<\phi_{\th},\phi_{\th}\right> 
				\end{array}\right) 
\ee
In particular, noting that
\[
\det\left(\mathcal{D}_\xi - \lambda(\xi)I_\xi\right) = (ik\xi)^3\det\left(\widehat{\mathcal{D}}_\xi - \mu(\xi)\widehat{I}_\xi\right)
\]
it follows that the $\mu_j(\xi)$ in \eqref{e:ev_factor} are precisely the eigenvalues of the $\widehat{\mathcal{D}}_\xi$.  
This establishes the following result.

\begin{theorem}
\label{t:pert}
Under the hypotheses of Theorem \ref{t:co-per}, 
the spectrum of the operator $\mathcal{A}[\phi]$ on $L^2(\RM)$ in a sufficiently small neighborhood of the origin consists of precisely three $C^1$
curves $\{\lambda_j(\xi)\}_{j=1}^3$ defined for $|\xi|\ll 1$ which can be expanded as
\[
\lambda_j(\xi) = ik\xi\mu_j(0)+o(\xi),~~~j=1,2,3
\]
where the $\mu_j(0)$ are the eigenvalues of the matrix $\widehat{\mathcal{D}}_0$ above.  As such, a necessary condition for the wave $\phi$ to be modulationally
stable in the sense of Definition \ref{d:MI} is that the eigenvalues $\mu_j(0)$ are all real.  Furthermore, the wave $\phi$ is modulationally
stable provided that the eigenvalues of $\widehat{\mathcal{D}}_0$ are real and distinct.
\end{theorem}

\begin{proof}
The necessity criteria above is clear since if any of the $\mu_j(0)$ have non-zero imaginary parts then the associated eigenvalue $\lambda_j(\xi)$ must have non-zero
real part for $0<|\xi|\ll 1$, implying modulational instability.  Now, suppose that the eigenvalues $\mu_j(0)$ of $\widehat{\mathcal{D}}_0$ are real and distinct.
Recalling from Remark \ref{r:spec_sym2} that the spectrum of the Bloch operators $\mathcal{A}_\xi[\phi]$ are symmetric about the imaginary axis,
it follows in this case that one must have  
\[
\lambda_j(\xi)\in\RM i
\]
for all $|\xi|\ll 1$ and each $j=1,2,3$, and hence the wave $\phi$ is modulationally stable in this case.
\end{proof}

The above  result provides a mathematically rigorous description of the modulational instability problem for the Camassa-Holm equation.
As described at the end of Section \ref{s:W}, however, Whitham's theory of modulations posits that the modulational instability of a given periodic
traveling wave $\phi$ is determined by the eigenvalues of the matrix $\mathcal{W}(\phi)$ defined in \eqref{e:D_Whitham}.  In the next section,
we reconcile these two results by proving they are in fact equivalent.

%{\color{red} 
%The big claim now is that the eigenvalues of $\widehat{\mathcal{D}}_0$ are real if and only if the eigenvalues of the linearized Whitham system about the background
%wave $\phi$ are real. In the next section we show that
%%
%\[
%W(\phi) = \widehat{\mathcal{D}}_0-cI.
%\]
%%
%Time now to calculate above inner products and compare to what comes from Whitham.
%}

\section{Proof of Theorem \ref{T:main}}\label{s:proof}

In this section, we establish Theorem~\ref{T:main} by performing a row-by-row computation to show that
\be
\label{e:equalmats}
\mathcal{W}(\phi) = \widehat{\mathcal{D}}_0-cI,
\ee
where $\mathcal{W}(\phi)$ is the matrix associated to the Whitham modulation equations \eqref{e:Wsys}, and the eigenvalues of $\widehat{\mathcal{D}}_0$, defined in \eqref{e:Dmat}, rigorously describe the structure of the $L^2(\R)$-spectrum of the linearized operator $\mathcal{A}[\phi]$ in a sufficiently small neighborhood of the origin: see Theorem~\ref{t:pert}.
In particular, this identity  implies that the eigenvalues of the the Whitham matrix $\mathcal{W}(\phi)$ and the matrix $\widehat{\mathcal{D}}_0$ have exactly the same imaginary
parts.\footnote{That is, the presence of the $cI$ term in \eqref{e:equalmats} does not effect the imaginary parts, hence stability predictions by Theorem \ref{t:pert}, for the eigenvalues.}

First, observe from \eqref{e:D_Whitham} and \eqref{e:Dmat} that the first rows of the matrices $\mathcal{W}(\phi)$ and $\widehat{\mathcal{D}}_0-cI$ are clearly identical.
To show that the  second rows are likewise identical, it is sufficient to establish the identities 
\be
\label{e:row2} 
\begin{cases}
\left< 1, A_1\phi_k + A_2\phi' \right> = \left< 1,-\frac{3}{2}\phi^2-\frac{1}{2}k^2\phi^2_{\theta}\right>_k, \\
\left< 1, A_1\phi_M \right> - c = \left< 1,-\frac{3}{2}\phi^2-\frac{1}{2}k^2\phi^2_{\theta}\right>_M, \\
\left< 1, A_1\phi_P \right> = \left< 1,-\frac{3}{2}\phi^2-\frac{1}{2}k^2\phi^2_{\theta}\right>_P.
\end{cases}
\ee
%%
%To this end, the following identities will be useful:
%%
%\be
%\label{e:G0indentities}
%G_0^{\dagger} = G_0, \qquad G_01 = 1, \qquad G_0\d_{\theta} = \d_{\theta}G_0,
%\ee
%%
%where, recall,
%%
%$$ G_0 = (1-k^2\d_\theta^2)^{-1}. $$
%%
We start by studying the second equation above.  First observe that
\[
\partial_\theta\left(1-k^2\partial_\theta^2\right)^{-1} = \left(1-k^2\partial_\theta^2\right)^{-1} \partial_\theta~~{\rm and }~~(1-k^2\partial_\theta^2)^{-1}1=1
\]
so that, recalling the definitions in \eqref{e:J_expand},  $J_01=J_0^\dag1=0$ and $J_11=J_1^\dag1=1$.  Recalling
the definitions in \eqref{e:Aidentities} and \eqref{e:J_expand}, we find that
%that $\left(J_0L_1\right)^\dag1=0$ and hence
%
\begin{align*}
\left< 1, A_1\phi_M\right> %&= \left< 1, (J_0L_1+J_1L_0)\phi_M\right> \\
%&= \left< 1,(J_0L_1)\phi_M\right> + \left< 1,(J_1L_0)\phi_M\right> \\
%&= \left<1, (J_1L_0)\phi_M\right> \\
%
%since $\left<1,(J_0L_1)\phi_M\right>=0$. Thus, using the operator identities and integrating by parts gives 
%
%\begin{align*} 
%\left<1,(J_1L_0)\phi_M\right> &=
%&= \left<1,G_0^2(1+k^2\d_{\theta}^2)L_0\phi_M\right> \\
&= \left< J_1^\dag 1,L_0\phi_M\right> \\
%&= \left< 1,L_0\phi_M\right> \\
&= \left< 1, (-k^2\d_{\theta}(c-\phi)\d_{\theta} + (c-3\phi+k^2\phi''))\phi_M\right> \\
%&= \left< 1, k^2\phi'\phi_M' - ck^2\phi_M'' + k^2\phi\phi_M'' + c\phi_M - 3\phi\phi_M + k^2\phi''\phi_M \right> \\
%&= \left< 1, k^2\phi'\phi_M' - ck^2\phi_M'' - k^2\phi'\phi_M' + c\phi_M - \thalf(\phi^2)_M - k^2\phi'\phi_M'\right> \\
&= \left< 1, -\half(k^2\phi_{\theta}^2)_M - \thalf(\phi^2)_M\right> + c\left<1,\phi_M\right>.
\end{align*} 
Hence, since $\langle 1,\phi_M\rangle = M_M=1$ it follows that
\be
\left<1,A_1\phi_M \right> - c = \left<1, -\thalf\phi^2 - \half k^2\phi_{\theta}^2 \right>_M
\ee
which establishes \eqref{e:row2} (ii). Next, noting that $\left<1,\phi_P\right>=P_M=0$, a completely analogous argument as above can be used to establish \eqref{e:row2} (iii). 

For \eqref{e:row2} (i), we begin by observing that
\begin{align*}
\left< 1,A_1\phi_k + A_2\phi'\right> %&= \left< 1,A_1\phi_k\right> + \left< 1,A_2\phi'\right> \\
%&= \left< 1,(J_0L_1+J_1L_0)\phi_k\right> + \left< 1,A_2\phi'\right> \\
&= \left< J_1^\dag 1,L_0\phi_k\right> + \left< 1,A_2\phi'\right>
\end{align*}
since, as above, $\left< 1,(J_0L_1)\phi_k\right>=0$.  Similar arguments to those used above show that
\[
\left<J_1^\dag 1,L_0\phi_k\right>=\left<1, -\thalf(\phi^2)_k - \half k^2(\phi')^2_k\right>.
\]
%Similarly, noting that $\left<1,\phi_k\right>=0$ a completely analogous argument as above shows that
%\be
%\left< 1,A_1\phi_k\right> = \left<1, -\thalf(\phi^2)_k - \half k^2(\phi')^2_k\right>.
%\ee 
%establishing \eqref{e:row2}(iii).
Continuing, recalling the definition of the operator $A_2$ from \eqref{e:Aidentities} we find that
\begin{align*}
\left< 1, A_2\phi'\right> %&= \left< 1, (J_0L_2+J_1L_1+J_2L_0)\phi'\right> \\
&= \left< J_0^{\dagger}1,L_2\phi'\right> + \left< 1,J_1L_1\phi'\right> + \left< 1,J_2L_0\phi'\right> \\
%&= -k\left< \d_{\theta}1,L_2\phi'\right> + \left< 1,(J_1L_1)\phi'\right> \\
&= \left< 1,J_1L_1\phi'\right>,
\end{align*}
where we have used that $J_0^\dag 1=0$ and $L_0\phi'=0$.  Further, using  that $J_1^\dag 1=1$ we find
\begin{align*}
\left< 1,J_1L_1\phi'\right> %&= \left< 1, (2k^2G_0^2\d_{\theta}^2 + G_0)L_1\phi'\right> \\
%&= \left< 1,(2k^2G_0^2\d_{\theta}^2)L_1\phi'\right> + \left< 1,(G_0L_1)\phi'\right> \\
%&= 2k^2\left< \d_{\theta}^21,L_1\phi'\right> + \left< G_0^{\dagger}1,L_1\phi'\right> \\
&= \left< 1,L_1\phi'\right> \\
%&= \left< 1, (-k\d_{\theta}((c-\phi)\cdot)-k(c-\phi)\d_{\theta})\phi'\right> \\
&= \left< 1, k(\phi')^2 - 2ck\phi'' + 2k\phi\phi'' \right> \\
%&= \left<1, -k(\phi')^2 - 2ck\phi''\right> \\
%&= \left<1, -k(\phi')^2\right>. 
&=-k\left<\phi',\phi'\right>.
\end{align*}
Thus,
\begin{align*}
\left<1, A_1\phi_k + A_2\phi'\right> %&= \left<1,A_1\phi_k\right> + \left<1,(J_1L_1)\phi'\right> \\
&= \left<1, -\thalf(\phi^2)_k - \half k^2(\phi')^2_k\right> -k \left<\phi',\phi'\right>
\end{align*}
which, by the product rule, is equivalent to
\be
\left<1, A_1\phi_k + A_2\phi'\right> = \left<1, -\thalf\phi^2 - \half k^2\phi_{\theta}^2\right>_k,
\ee
which establishes \eqref{e:row2} (i), and hence the second rows of the matrices in \eqref{e:equalmats} are identical.

Finally, to compare the third row of the matrices in \eqref{e:equalmats} it is sufficient to establish the identities
\be
\label{e:row3}
\begin{cases}
%\left< G_0^{-1}\phi, A_1\phi_k+A_2\phi'+kc_k\phi_k \right> = \left<1, ck^2\phi_{\theta}^2-\phi^3+k^2\phi^2\phi_{\theta\theta}\right>_k, \\
%\left< G_0^{-1}\phi, A_1\phi_M\right> - k^2c_M\left< \phi',\phi'\right> = \left<1, ck^2\phi_{\theta}^2-\phi^3+k^2\phi^2\phi_{\theta\theta}\right>_M, \\
%\left<G_0^{-1}\phi,A_1\phi_P\right> - k^2c_P\left< \phi',\phi'\right> - c = \left< 1, ck^2\phi_{\theta}^2 - \phi^3 + k^2\phi^2\phi_{\theta\theta}\right>_P. 
\left< \left(1-k^2\partial_\theta^2\right)\phi, A_1\phi_k+A_2\phi'+kc_k\phi_k \right> = \left<1, ck^2\phi_{\theta}^2-\phi^3+k^2\phi^2\phi_{\theta\theta}\right>_k, \\
\left< \left(1-k^2\partial_\theta^2\right)\phi, A_1\phi_M\right> - k^2c_M\left< \phi',\phi'\right> = \left<1, ck^2\phi_{\theta}^2-\phi^3+k^2\phi^2\phi_{\theta\theta}\right>_M, \\
\left<\left(1-k^2\partial_\theta^2\right)\phi,A_1\phi_P\right> - k^2c_P\left< \phi',\phi'\right> - c = \left< 1, ck^2\phi_{\theta}^2 - \phi^3 + k^2\phi^2\phi_{\theta\theta}\right>_P. 
\end{cases}
\ee 
These identities follow by similar means as above, albeit the computations are more involved. We begin by studying \eqref{e:row3} (ii).  To this end,
using \eqref{e:Aidentities} we first note that
\begin{equation}\label{e:A1_1}
\left< \left(1-k^2\partial_\theta^2\right)\phi, A_1\phi_M\right> %&= \left< G_0^{-1}\phi, (J_0L_1+J_1L_0)\phi_M\right> \\
= \left< \left(1-k^2\partial_\theta^2\right)\phi, J_0L_1\phi_M\right> + \left< \left(1-k^2\partial_\theta^2\right)\phi, J_1L_0\phi_M\right>. 	
\end{equation}
A direct calculation using the definition of $J_1$ in \eqref{e:J_expand} shows  that
%Thus, since $G_0\d_{\theta} = \d_{\theta}G_0$, and $\mathcal{A}[\phi]\phi_M = -kc_M\phi'$ (see \eqref{e:AphiM}) it follows that 
%
\begin{align*}
\left< \left(1-k^2\partial_\theta^2\right)\phi, J_1L_0\phi_M\right> 
%%&= \left< G_0^{-1}\phi, (2k^2G_0^2\d_{\theta}^2)L_0\phi_M\right> + \left< G_0^{-1}\phi, G_0L_0\phi_M\right> \\
%&= \left< G_0^{-1}\phi, 2k(G_0\d_{\theta})(J\mathcal{L}[\phi])\phi_M\right> + \left< G_0^{-1}\phi, G_0L_0\phi_M\right> \\
%&= 2\left< (k\d_{\theta})^{\dagger}G_0^{\dagger}G_0^{-1}\phi, \mathcal{A}[\phi]\phi_M \right> + \left< G_0^{-1}\phi, G_0L_0\phi_M\right> \\
%&= -2k\left<\phi', \mathcal{A}[\phi]\phi_M\right> + \left< G_0^{-1}\phi, G_0L_0\phi_M\right> \\
%&= 2k^2c_M\left< \phi',\phi'\right> + \left<\phi, L_0\phi_M\right>.
&=-2k\left<\phi', \mathcal{A}[\phi]\phi_M\right>+ \left<\phi, L_0\phi_M\right>.
\end{align*}
Moreover, integrating by parts gives
\begin{align*}
\left<\phi, L_0\phi_M\right> &= \left<\phi, k^2\phi'\phi_M' - ck^2\phi_M'' + k^2\phi\phi_M'' + c\phi_M - 3\phi\phi_M + k^2\phi''\phi_M\right> \\
&= \left<\phi, k^2\phi'\phi_M' + k^2\phi\phi_M'' - 3\phi\phi_M + k^2\phi''\phi_M\right> + c\left<\phi-k^2\phi'',\phi_M\right>.
%= \left<\phi, k^2\phi'\phi_M' + k^2\phi\phi_M'' - 3\phi\phi_M + k^2\phi''\phi_M\right>,
%&=\left<1,\frac{k^2}{2}\phi^2\phi_M''-3\phi^2\phi_M+k^2\phi\phi''\phi_M\right>
\end{align*}
Using that\footnote{See also Theorem \ref{t:co-per}.}
\[
\left<\phi-k^2\phi'', \phi_M\right>  =P_M=0,
\]
integration by parts, as well as \eqref{e:AphiM} it follows that
\begin{align*}
\left< \left(1-k^2\partial_\theta^2\right)\phi, J_1L_0\phi_M\right> &=2k^2c_M\left< \phi',\phi'\right>%+\left<\phi, k^2\phi'\phi_M' + k^2\phi\phi_M'' - 3\phi\phi_M + k^2\phi''\phi_M\right>.
+\left<1,\frac{k^2}{2}\phi^2\phi_M''-3\phi^2\phi_M+k^2\phi\phi''\phi_M\right>\\
&=2k^2c_M\left< \phi',\phi'\right>+\left<1,\frac{k^2}{2}\left(\phi^2\phi''\right)_M-\left(\phi^3\right)_M\right>.
\end{align*}
Similarly, using integration by parts we find
\begin{align*}
\left<\left(1-k^2\partial_\theta^2\right)\phi, J_0L_1\phi_M\right> %&= -k\left<\phi',L_1\phi_M\right> \\
&= k\left<\phi,\partial_\theta L_1\phi_M\right> \\
%&= -k\left<\phi',(-k\d_{\theta}((c-\phi)\cdot)-k(c-\phi)\d_{\theta})\phi_M\right> \\
%&= -k\left<\phi', k\phi'\phi_M - 2ck\phi_M' + 2k\phi\phi_M' \right> \\
&= \left<\phi, k^2\phi''\phi_M + 3k^2\phi'\phi_M' - 2ck^2\phi_M'' + 2k^2\phi\phi_M'' \right>\\
%&= k^2c\left<\phi'\phi'\right>_M+\left<1, k^2\phi\phi''\phi_M -\frac{3k^2}{2}\phi^2\phi_M''  + 2k^2\phi^2\phi_M'' \right> \\
&= k^2c\left<\phi'\phi'\right>_M+\left<1, k^2\phi\phi''\phi_M +\frac{k^2}{2}\phi^2\phi_M''  \right> \\
&=k^2c\left<\phi'\phi'\right>_M+\left<1,\frac{k^2}{2}\left(\phi^2\phi''\right)_M\right>.
\end{align*}
From \eqref{e:A1_1}, it follows that
\begin{align*}
\left<\left(1-k^2\partial_\theta^2\right)\phi,A_1\phi_M\right> &= 
	k^2c_M\left<\phi',\phi'\right>+\left<1,k^2c\left(\phi'\right)^2-\phi^3+k^2\left(\phi^2\phi''\right)\right>_M,
% + \left<\phi, 4k^2\phi'\phi_M' + 3k^2\phi\phi_M'' - 3\phi\phi_M + 2k^2\phi''\phi_M - 2ck^2\phi_M''\right> \\
%&= 2k^2c_M\left<\phi',\phi'\right> + \int_0^1 k^2\phi^2\phi_M'' - (\phi^3)_M + k^2(\phi^2)_M\phi'' + ck^2(\phi')^2_M\,d\theta \\
%&= 2k^2c_M\left<\phi',\phi'\right> + \int_0^1 (k^2\phi^2\phi'')_M + ck^2(\phi')^2_M\,d\theta. 
\end{align*}
%%
%Thus,
%%
%\be
%\left<G_0^{-1}\phi,A_1\phi_M\right> - k^2c_M\left<\phi',\phi'\right> = \left<1, ck^2\phi_{\theta}^2 - \phi^3 + k^2\phi^2\phi_{\th\th}\right>_M
%\ee 
%%
which establishes \eqref{e:row3} (ii).  Furthermore, using the fact that\footnote{This can also be seen from Theorem \ref{t:co-per}.}
\[
\left<\phi-k^2\phi'',\phi_P\right>=P_P=1,
\]
along with \eqref{e:AphiP}, a completely analogous argument as above establishes \eqref{e:row3} (iii). 

Finally, it remains to establish \eqref{e:row3} (i).  To this end, we begin by noting that
\begin{align*}
%\left<G_0^{-1}\phi, A_1\phi_k + A_2\phi' + kc_k\phi_k\right> %&= \left<G_0^{-1}\phi, A_{1}\phi_k\right> + \left<G_0^{-1}\phi, A_2\phi'\right> + kc_k\left<G_0^{-1}\phi, \phi_k\right> \\
%&= \left<G_0^{-1}\phi,(J_1L_0+J_0L_1)\phi_k\right> + \left<G_0^{-1}\phi, A_2\phi'\right> + kc_k\left<G_0^{-1}\phi, \phi_k\right>.
\left<\left(1-k^2\partial_\theta^2\right)\phi, A_1\phi_k + A_2\phi' + kc_k\phi_k\right>
&= \left<\left(1-k^2\partial_\theta^2\right)\phi,(J_1L_0+J_0L_1)\phi_k\right>\\
&\quad+ \left<\left(1-k^2\partial_\theta^2\right)\phi, A_2\phi'\right>  + kc_k\left<\left(1-k^2\partial_\theta^2\right)\phi, \phi_k\right>.
\end{align*}
Observing that
\[
\left<\phi-k^2\phi'',\phi_k\right> = P_k-k\left<\phi',\phi'\right>
\]
and that $P_k=0$, along with \eqref{e:Aphik} and the fact that $L_0\phi'=0$, calculations completely analogous to the above yield 
%Begin by noting that since $\left<\phi-k^2\phi'',\phi_k\right> = -k\left<\phi',\phi'\right>$, $\mathcal{A}[\phi]\phi_k = -kc_k\phi'-A_1\phi'$ (see \eqref{e:Aphik}), and $L_0\phi'=0$ it follows that
\begin{align*}
 \left<\left(1-k^2\partial_\theta^2\right)\phi,J_1L_0\phi_k\right>
%&= (2k^2c_k-kc)\left<\phi',\phi'\right> + 2k\left<\phi',J_0L_1\phi'\right>  + \int_0^1 {\color{blue}k^2\phi\phi'\phi'_k} + k^2\phi^2\phi_k'' - (\phi^3)_k + k^2\phi\phi''\phi_k\,d\theta\\
%&=(2k^2c_k-kc)\left<\phi',\phi'\right> + 2k\left<\phi',(J_0L_1)\phi'\right> \\
%&\qquad\qquad+\left<1,\frac{k^2}{2}\left(\phi^2\phi''\right)_k-\left(\phi^3\right)_k+\frac{k^2}{2}\phi^2\phi_k''\right>
&=(2k^2c_k-kc)\left<\phi',\phi'\right> + 2k\left<\phi',J_0L_1\phi'\right> \\
&\qquad\qquad+\left<1,\frac{k^2}{2}\left(\phi^2\phi''\right)_k-\left(\phi^3\right)_k\right>
\end{align*}
and
\begin{align*}
\left<\left(1-k^2\partial_\theta^2\right)\phi,J_0L_1\phi_k\right> 
%&= -k^2\left<\phi',\phi'\phi_k\right> + 2ck^2\left<\phi',\phi_k'\right> - 2k^2\left<\phi',\phi\phi_k'\right>. \\
%&=k^2\left<\phi,\phi''\phi_k+\phi'\phi_k'\right>+ck^2\left<\phi',\phi'\right>_k-2k^2\left<\phi',\phi\phi_k'\right>\\
%&=\left<1,\frac{k^2}{2}\left(\phi^2\right)_k\phi''\right>+ck^2\left<\phi',\phi'\right>_k-k^2\left<\phi',\phi\phi_k'\right>\\
%&=\left<1,\frac{k^2}{2}\left(\phi^2\right)_k\phi''\right>+ck^2\left<\phi',\phi'\right>_k+\left<1,\frac{k^2}{2}\phi^2\phi_k''\right>\\
&=\left<1,\frac{k^2}{2}\left(\phi^2\phi''\right)_k\right>+ck^2\left<\phi',\phi'\right>_k.
\end{align*}
Next, using that $L_0\phi'=0$ and \eqref{e:Aidentities} we have that
\begin{align*}
\left<\left(1-k^2\partial_\theta^2\right)\phi, A_2\phi'\right> %&= \left<G_0^{-1}\phi, (J_0L_2)\phi'\right> + \left<G^{-1}\phi, (J_1L_1)\phi'\right> + \left<G_0^{-1}\phi, (J_2L_0)\phi'\right> \\
&= \left<\left(1-k^2\partial_\theta^2\right)\phi, J_0L_2\phi'\right> + \left<\left(1-k^2\partial_\theta^2\right)\phi, J_1L_1\phi'\right>,
\end{align*}
where, by calculations similar to those above, we compute that
\begin{align*}
\left< \left(1-k^2\partial_\theta^2\right)\phi, J_0L_2\phi'\right> &= -k\left<\phi',L_2\phi'\right> \\
%&= -k\left<\phi', -(c-\phi)\phi'\right> \\
%&= -k\left<\phi', -c\phi' + \phi\phi'\right> \\
&= ck\left<\phi',\phi'\right> - k\left<\phi', \phi\phi'\right>, 
\end{align*}
and
\begin{align*}
\left<\left(1-k^2\partial_\theta^2\right)\phi, (J_1L_1)\phi'\right> %&= \left<G_0^{-1}\phi, (2k^2G_0^2\d_{\theta}^2 + G_0)L_1\phi'\right> \\
%&= \left<G_0^{-1}\phi, (2k^2G^2\d_{\theta}^2)L_1\phi'\right> + \left<G_0^{-1}\phi, (G_0L_1)\phi'\right> \\
%&= 2\left<G_0^{-1}\phi, k^2G_0\d_{\theta}G_0\d_{\theta}(L_1\phi')\right> + \left< \phi, L_1\phi'\right> \\
%&= -2k\left<\phi', (J_0L_1)\phi'\right> + \left<\phi, L_1\phi'\right> \\
%&= -2k\left< \phi', (J_0L_1)\phi'\right> + \left<\phi, (-k\d_{\theta}((c-\phi)\cdot)-k(c-\phi)\d_{\theta})\phi'\right> \\
%&= -2k\left<\phi', (J_0L_1)\phi'\right> + \left<\phi, k(\phi')^2 - 2ck\phi'' + 2k\phi\phi''\right> \\
%&= -2k\left<\phi', (J_0L_1)\phi'\right> + k\left< \phi,(\phi')^2\right> - 2ck\left<\phi, \phi''\right> + 2k\left<\phi, \phi\phi''\right>. 
&= -2k\left<\phi', (J_0L_1)\phi'\right> + 2ck\left<\phi', \phi'\right> +k\left<\phi', \phi\phi'\right>+2k\left<1,\phi^2\phi''\right>.
\end{align*}
%
%Noting that integration by parts gives that
%\[
%\left<\phi',\phi\phi'\right>=-\frac{1}{2}\left<1,\phi^2\phi''\right>,
%\]
%%Note that
%%%
%%$$-k\left<\phi',\phi\phi'\right> = -k\left<\phi, -(\phi\phi')'\right> = -k\left<\phi, -(\phi')^2 - \phi\phi''\right> = k\left<\phi,(\phi')^2\right> + k\left<\phi, \phi\phi''\right>. $$
%%%
Thus,
\begin{align*}
\left<(1-k^2\partial_\theta^2)\phi,A_2\phi'\right>&=3ck\left<\phi',\phi'\right>+2k\left<1,\phi^2\phi''\right>-2k\left<\phi',J_0L_1\phi'\right>.
\end{align*}
Noting also that 
\begin{align*}
\left<(1-k^2\partial_\theta^2)\phi, kc_k\phi_k\right> %&= kc_k\left<(1-k^2\partial_\theta^2)phi, \phi_k\right> \\
%&= kc_k \left<\phi-k^2\phi'', \phi_k\right> \\
&= -k^2c_k\left<\phi', \phi'\right>
\end{align*}
putting everything together yields
\begin{align*}
\left<(1-k^2\partial_\theta^2)\phi, A_1\phi_k + A_2\phi' + kc_k\phi_k\right> &=\left(c_kk^2+2ck\right)\left<\phi',\phi'\right>+ck^2\left<\phi',\phi'\right>_k\\
&\qquad+\left<1,k^2(\phi^2\phi'')_k+2k\phi^2\phi''-(\phi^3)_k\right>\\
&=\left<1,ck^2(\phi')^2-\phi^3-k^2\phi^2\phi''\right>_k,
\end{align*}
which establishes \eqref{e:row3} (i), thus completing the proof of Theorem \ref{T:main}.

\appendix
\section{Stokes Waves for the Camassa-Holm Equation}\label{a:stokes}

In this section, we investigate the Stokes wave solutions of the CH equation \eqref{e:ch}, which here correspond to equilibrium solutions 
of the rescaled evolutionary equation \eqref{e:stationary_rescale} with asymptotically small oscillations about their mass\footnote{Recall by Remark \ref{r:mass} 
that the mass $M$ of a periodic traveling wave solution of the CH equation is necessarily positive.} $M>0$.  While waves have already
been shown to exist in Section \ref{s:existence}, our goal here is to provide an analytic parameterization of these waves.

To this end, we note that such Stokes waves are necessarily 
%equilibrium solutions of the rescaled evolutionary equation \eqref{e:stationary_rescale}, i.e. they are 
$1$-periodic solutions of the rescaled profile equation
\begin{equation}\label{e:profile_rescale3}
-\omega\left(\phi'-k^2\phi'''\right)+3k\phi\phi'=k^3\left(2\phi'\phi''+\phi\phi'''\right),
\end{equation}
where here primes denote differentiation with respect to the traveling variable $\theta=kx-\omega t$ and $\omega=kc$.  We now let
$F:H^3_{\rm per}(0,1)\times\RM^+\times\RM\mapsto L^2_{\rm per}(0,1)$ be defined via
\[
F(\phi;k,\omega):=-\omega\left(\phi'-k^2\phi'''\right)+3k\phi\phi'-k^3\left(2\phi'\phi''+\phi\phi'''\right),
\]
and note that for each $k>0$ and $\omega\in\RM$ solutions of \eqref{e:profile_rescale3} correspond to solutions $\phi\in H^3_{\rm per}(0,1)$ of
\[
F(\phi;k,\omega)=0.
\]
Observe that since  \eqref{e:profile_rescale3} is invariant under the transformations
\[
\theta\mapsto \theta+\theta_0~~{\rm and}~~\theta\mapsto-\theta
\]
we can without loss of generality seek solutions $\phi\in H^3_{\rm per}(0,1)$ of \eqref{e:profile_rescale3} that are even in $\theta$.  Further,
we note that $F$ is an analytic function of its arguments.

Now, for each $M>0$ we clearly have $F(M;k,\omega)=0$ for all $k>0$ and $\omega\in\RM$.  Using the Implicit Function Theorem, we know that non-constant
solutions of \eqref{e:profile_rescale3} may bifurcate from $\phi=M$ for a fixed $k>0$ provided that $\omega\in\RM$ is chosen so that the linearization
\[
D_\phi F(M;k,\omega)=\left(3kM-\omega\right)\partial_\theta+\left(\omega k^2-k^3M\right)\partial_\theta^3
\]
is not an isomorphism from $H^3_{\rm per}(0,1)$ into $L^2_{\rm per}(0,1)$.  Observing that
\[
%D_\phi F(M;k,\omega)\cos(2\pi n\theta) = 2\pi n\left[\omega\left(1+4k^2n^2\pi^2\right)-\left(3kM-4k^3Mn^2\pi^2\right)\right]\sin(2\pi n\theta)
D_\phi F(M;k,\omega)e^{2\pi i n\theta} = 2\pi  in\left[\omega\left(1+4k^2n^2\pi^2\right)-\left(3kM-4k^3Mn^2\pi^2\right)\right]e^{2\pi i n\theta}
\]
for all $n\in\ZM$, it follows that 
\[
1,e^{\pm 2\pi i x}\in \ker\left(D_\phi F(M;k,\omega_0)\right)
\]
provided that 
%$\cos(2\pi z)$ belongs to the kernel of $D_\phi F(M;k,\omega)$ provided that
$\omega=\omega_0(k,M)$, where
\[
\omega_0(k,M):=\frac{3kM+4k^3M\pi^2}{1+4k^2\pi^2}.
\]
Further, noting that the function
\[
(0,\infty)\ni k\mapsto\omega_0(k,M)
\]
is strictly monotone it follows that for all $M,k>0$, we have that
\[
\ker\left(D_\phi F(M;k,\omega_0)\right) =% {\rm span}\left\{\cos(2\pi\theta)\right\}.
{\rm span}\left\{1, e^{\pm 2\pi i x}\right\}.
\]
Since $D_\phi F(M;k,\omega_0)$ is clearly skew adjoint, it follows that the co-kernel of $D_\phi F(M;k,\omega_0)$ is also three-dimensional, and hence the linear operator
$D_\phi F(M;k,\omega_0)$ is necessarily Fredholm with index zero for all $M,k>0$.  

Using a straightforward Lyapunov-Schmidt argument, one may now construct for each $M,k>0$ a one-parameter family of non-constant, even, smooth solutions
of \eqref{e:profile_rescale3} near the bifurcation point $(\phi,\omega)=(M,\omega_0)$.  For details, see \cite{BT03,HP17,J13,K04}, for example.  
%one can now prove for each $k>0$ 
%the existence of a family of non-constant, even, smooth solutions of \eqref{e:profile_rescale2} near the bifurcation point $(M,\omega_0(k,M))$.  
In particular, for each $M,k>0$ there exists a family of $1$-periodic solutions $\phi(\theta;k,M,A)$ of \eqref{e:profile_rescale3}  with
temporal frequency $\omega(k,M,A)$, defined for $|A|\ll 1$, where here $\phi$ and $\omega$ depend analytically on $A$ and $\phi$
is smooth and even in $\theta$ and, furthermore, $\omega$ is even in $A$.  Furthermore, for $|A|\ll 1$ the functions $\phi$ and $\omega$
admit asymptotic expansions of the form
\[
\left\{\begin{aligned}
\phi(\theta;k,M,A) &= M  + A\cos\left(2\pi\theta\right) + \sum_{j=2}^\infty A^j\psi_j(\theta;k,M),\\
\omega(k,M,A)&=\omega_0+\sum_{j=1}^\infty\omega_{2j}(k,M)A^{2j},
\end{aligned}\right.
\]
where  
\begin{equation}\label{e:FA_satisfy}
\int_0^1\psi_j(\theta)d\theta=\int_0^1\psi_j(\theta)\cos(2\pi\theta)d\theta = 0
\end{equation}
for all $j\geq 2$.  

Substituting these expansions into \eqref{e:profile_rescale3} yields an infinite hierarchy of equations indexed by the order of the small parameter $A$.  By construction,
the first non-trivial equation occurs at $\mathcal{O}(A^2)$, which reads
\[
D_\phi F(M;k,\omega_0)\psi_2 = 3\pi k\left(1+4k^2\pi^2\right)\sin\left(4\pi x\right).
\]
By the Fredholm Alternative, this has a $1$-periodic solution satisfying  \eqref{e:FA_satisfy} given by
\[
\psi_2(\theta;k,M)= \frac{\left(1+4k^2\pi^2\right)^2}{32k^2M\pi^2}\cos\left(4\pi\theta\right).
\]
Continuing, the $\mathcal{O}(A^3)$ equation reads
\[
D_\phi F(M;k,\omega_0)\psi_3 = 
%\frac{1+4k^2\pi^2}{16kM\pi}\left[6+72k^2\pi^2+192k^4\pi^4-32kM\pi^2\omega_2+3\left(3+40k^2\pi^2+112k^4\pi^4\right)\cos\left(4\pi x\right)\right]\sin(2\pi\theta).
\frac{1+4k^2\pi^2}{16kM\pi}\left[\mathcal{R}_1+\mathcal{R}_2\cos\left(4\pi x\right)\right]\sin(2\pi\theta)
\]
where
\[
\mathcal{R}_1=6+72k^2\pi^2+192k^4\pi^4-32kM\pi^2\omega_2,~~\mathcal{R}_2=3\left(3+40k^2\pi^2+112k^4\pi^4\right).
\]
By the Fredholm Alternative again, it follows that this has a $\pi$-periodic solution provided that
\[
\omega_2(k,M)=\frac{3\left(1+4k^2\pi^2\right)^2}{64kM\pi^2}.
\]
Together, the above arguments provide the analytic Stokes wave expansions for the CH equation given in \eqref{e:small_expand}.

\bibliographystyle{abbrv}
\bibliography{CH}

\end{document}